\documentclass[leqno,a4paper,10pt]{amsart}

\usepackage{amsmath, amsfonts, amsthm, amssymb, dsfont}

\usepackage{times}
\usepackage{microtype}
\usepackage{cite}

\usepackage{mathrsfs}
\usepackage{enumerate, xspace}
\usepackage{graphicx}
\usepackage{color}
\usepackage[marginpar]{todo}

\usepackage[all, knot]{xy}
\xyoption{arc}

\usepackage{pgf,tikz}
\usepackage{mathrsfs}
\usetikzlibrary{arrows}

\usepackage{verbatim}

\usepackage{caption}
\usepackage{subcaption}

\usepackage[
	pdfauthor={ Markus Steenbock},
	pdftitle={Subgroups and diversity of left-orderable small cancellation groups},
	pdfkeywords={left-ordered groups, small cancellation theory, hyperbolic groups, quasi-isometric diversity},
	pdftex, 
	breaklinks
]{hyperref}

\hypersetup{
	colorlinks=true,%
	citecolor=black,%
	filecolor=black,%
	linkcolor=black,%
	urlcolor=black,%
}

\theoremstyle{plain}

\newtheorem{thm}{Theorem}[section]
\newtheorem{prop}[thm]{Proposition}
\newtheorem{lem}[thm]{Lemma}		
\newtheorem{cor}[thm]{Corollary}

\theoremstyle{definition}

\newtheorem{ex}[thm]{Example}	
				
\newtheorem{rem}[thm]{Remark}

\theoremstyle{remark}

\title{Subgroups and diversity of left-orderable small cancellation groups}
\author{Markus Steenbock}
\address{Fakult\"at f\"ur Mathematik, Universit\"at Wien,  1090 Wien, Austria}
\email{markus.steenbock@univie.ac.at}

\date{\today}


\begin{document}

\begin{abstract}
We arrange classical small cancellation constructions to produce left-orderable groups: we show that every finitely generated group is
 the quotient of a  left-ordered small cancellation group by a finitely generated kernel (Rips construction). We give presentations of
  left-ordered hyperbolic small cancellation groups that are not locally indicable, and observe that the class of left-ordered small
   cancellation groups that are not locally indicable is quasi-isometrically diverse. Altogether, this shows that left-orderable small
    cancellation groups form a rich and diverse class.
\end{abstract}

\maketitle

% Index
%\tableofcontents

\section{Introduction}

A group is left-orderable if it admits a total order that is invariant under left-multiplication. Equivalently, a countable group is left-orderable if it acts by orientation preserving homeomorphisms on the real line. Left-orderable groups attract a lot of attention in the context of dynamics and low dimensional topology, see e.g. \cite{deroin_groups_2014,clay_ordered_2016}. 
We are interested in  left-orderability of negatively curved groups, such as hyperbolic groups or small cancellation groups. Many natural examples of such groups are left-orderable, for instance, non-abelian free groups, surface groups, one relator quotients of such groups \cite{brodskii_equations_1980,hempel_one-relator_1990,rolfsen-free-2001,antolin_non-orientable_2011} and the fundamental groups of some hyperbolic $3$-manifolds \cite{boyer_orderable_2005}, see also \cite{clay_ordered_2016}. Negatively curved left-orderable groups can also be produced as amalgamated free products or HNN-extensions of such groups under appropriate compatibility conditions \cite[Theorems A and B]{bludov_word_2009}, and as graphs of groups under the appropriate conditions  \cite{chiswell_right_2011}.  
 However, a `typical' negatively curved group is not expected to be left-orderable. Indeed, random groups are torsion-free but not left-orderable in the density as well as the triangular model \cite{orlef_random_2017,orlef-non-orderability_2021}. In addition, every torsion-free small cancellation construction can easily be made not left-orderable by adding two relations to the presentation of the respective group, see Example~\ref{E:non-lo} below. Thus, there is a wealth of torsion-free finitely presented small cancellation groups that are not left-orderable, even though every finitely presented small cancellation group is commensurable to a subgroup of a right-angeled Artin group \cite{agol_virtual_2013,haglund_special_2008,wise_cubulating_2004}, hence, contains a finite index bi-orderable subgroup.
 
In this paper we study left-orderable small cancellation groups. Despite the fact that a typical finitely presented small cancellation group is expected not to be left-orderable, we show that left-orderable small cancellation groups form a rich class of groups, see Sections \ref{SI:rips}, \ref{SI:locally-indicable} and \ref{SI:diversity} below. To do so, we use some classical small cancellation constructions to produce left-orderable groups that split as a free product with amalgamation or as an HNN-extension of free groups. In this process, we choose explicit families of left-orders on the free group, see \cite{sunic_explicit_2013}, and arrange the constructions so that left-orderability can be proven using the respective combination theorem for left-ordered groups \cite[Theorems A and B]{bludov_word_2009}.

\subsection{Subgroups of left-orderable small cancellation groups}\label{SI:rips} In this paper, a small cancellation group is a group given by a group presentation $\langle X\mid R\rangle$, where $X$ is a set of generators and $R$ a set of relations that satisfy the $C'(1/6)$-condition in the sense of \cite[Section V.2]{lyndon_combinatorial_2001}. 
The main result of this paper is a left-orderable version of Rips' construction \cite{rips_subgroups_1982}. 
\begin{thm}\label{IT:Rips}
Let $Q$ be a finitely generated group. Then there is a short exact sequence $$ 1 \rightarrow N \rightarrow G \rightarrow Q\rightarrow 1$$ of groups 
 such that 
\begin{itemize}
\item $G$ is a left-orderable finitely generated small cancellation group,
\item the kernel $N$ is finitely generated. 
\end{itemize}
Moreover, if $Q$ is finitely presented, then $G$ can be chosen finitely presented. 
\end{thm}

\begin{rem}
As a finitely presented small cancellation group is hyperbolic, Theorem \ref{IT:Rips} yields  left-orderable and hyperbolic groups $G$ whenever $Q$ is finitely presented. In general, the group $G$ is a direct limit of hyperbolic groups. 
\end{rem}

\begin{rem} Every finitely generated recursively presented group is a quotient by  a finitely generated subgroup of a subgroup of a right-angled Artin group, hence, a bi-orderable group \cite[Proposition 7.2]{antolin_tits_2015}. However, in this construction the bi-orderable group is not necessarily a small cancellation or hyperbolic group, although it is a subgroup of a finitely presented small cancellation group, see \cite[Proposition 7.2]{antolin_tits_2015} and \cite[Theorem 10.1]{haglund_special_2008}. 
\end{rem}

\begin{rem}\label{IR:Rips1} If $Q$ is $n$-generated, then we construct $G$ to be $2n+10$ generated and $N$ to be $n+10$ generated. In the original Rips construction $N$ can be made $2$-generated \cite{rips_subgroups_1982}.
\end{rem} 

Rips' construction is a prominent tool to construct examples of negatively curved groups with various algebraic, geometric and algorithmic properties that contrast the properties of free or surface groups. This is often done by using the short exact sequence (see Theorem \ref{IT:Rips}) to lift properties of the quotient $Q$ to $G$ or $N$. In view of Theorem \ref{IT:Rips}, this method can be used to produce left-orderable small cancellation groups with the respective properties. For example, there are appropriate left-orderable versions of \cite{rips_subgroups_1982} (properties of subgroups and algorithmic properties), \cite{baumslag_unsolvable_1994}, \cite{bridson_malnormality_2001} (algorithmic and geometric properties), \cite{gitk_growth_2020} and \cite[Theorem 1.3]{coulon_examples_2022} (properties of growth functions and automorphisms). Thus, the properties of left-orderable small cancellation groups encompass a wide range of properties of hyperbolic groups that are not satisfied by free or surface groups. In this sense, the left-orderable small cancellation groups, and the left-orderable hyperbolic groups, are rich classes of negatively curved groups.     

\begin{rem} It remains open, for example, whether there is a left-orderable version of \cite{bumagin_every_2005} and, thus, \cite[Theorem 1.4]{coulon_examples_2022}. In fact, the existence of a short exact sequence as in Theorem \ref{IT:Rips} is not sufficient to prove these results. Also, the cohomological dimension of the groups $G$ of Theorem \ref{IT:Rips} is at most $2$. See \cite{arenas_cubical_2022} for a Rips construction that allows for arbitrarily large cohomological dimension. 
\end{rem}

\begin{rem}  As the groups in Theorem \ref{IT:Rips} are left-orderable they have the unique product property. Thus Theorem \ref{IT:Rips} is in a stark contrast to the Rips construction without unique product of \cite{arzhantseva_rips_2023}. Let us also note  that it is unknown whether or not small cancellation groups have the unique product property, see \cite[Problem 11.40 of Ivanov]{khukhro_unsolved_2023}.
\end{rem}

To prove Theorem \ref{IT:Rips}, we build the groups $G$ as an HNN-extension of a free group. This idea was previously used to produce residually finite groups $G$ \cite{wise_residually_2003}. Thus, the main point in the proof of Theorem \ref{IT:Rips} is to choose a concrete family of left-orders on the free group and to arrange the relators of $G$ so that the assumptions of the respective combination theorem, \cite[Theorem B]{bludov_word_2009}, can be verified. 

\subsection{Groups that are not locally indicable}\label{SI:locally-indicable}

A group is \emph{locally indicable} if every non-trivial finitely generated subgroup admits the integers as a homomorphic image. Every locally indicable group is left-orderable. It turns out that left-orderable small cancellation groups are a rich source for left-orderable groups that are not locally indicable. 

\begin{prop}\label{IP:not-locally-indicable} For every finitely generated group $Q$, the groups $G$ and  $N$ in Theorem \ref{IT:Rips} can be chosen to be left-orderable and not locally indicable.
\end{prop}

To produce left-orderable groups that are not locally indicable, we give finite small cancellation presentations of such groups and appeal to the left-orderable combination theorems \cite{bludov_word_2009} to show that they are left-orderable. We then arrange to embed these groups in the kernel $N$ in Theorem \ref{IT:Rips}.

\subsection{Quasi-isometric diversity}\label{SI:diversity}

Bowditch produced continuum many quasi-isometry classes of finitely generated small cancellation groups  \cite{bowditch_continuously_1998}. Bowditch's argument and \cite[Theorem A]{bludov_word_2009} 
yield continuum many quasi-isometry classes of finitely generated left-orderable small cancellation groups, see Example \ref{E:Bowditch} below. We combine the presentations of these groups with small cancellation presentations of left-orderable groups that are not locally indicable. This yields the following. 

\begin{prop}\label{IP:qi-diversity} There are continuum many, up to quasi-isometry, left-orderable small cancellation groups that are not locally indicable. 
\end{prop}
\begin{rem}\label{IR:Cantor}
As the space of left-orders of a free product of left-ordered groups has no isolated points \cite{rivas_left-orderings_2012}, we can arrange for the spaces of left-orders of the groups of Proposition \ref{IP:qi-diversity} to be homeomorphic to a  Cantor set, see Remark \ref{R:Cantor} below. In particular, there are continuum many, up to quasi-isometry, groups whose space of left-orders is a Cantor set, see Corollary \ref{C:Bowditch} below. It remains open whether there is an uncountable such family of one-ended groups. 
\end{rem}

\subsection{Perfect groups}

A group is perfect if its abelianisation is trivial. We note that perfect groups are not locally indicable, and that every simple group is perfect. The class of finitely generated simple left-orderable groups is now a vast and well studied class of groups, see e.g.  \cite{hyde_finitely_2019,mattebon_groups_2020,hyde_uniformly_2021}. For example, there are continuum many finitely generated simple left-orderable groups up to isomorphism \cite[Corollary 1.5]{hyde_finitely_2019}. In particular, there are continuum many finitely generated perfect left-orderable groups up to isomorphism. 

 We prove quasi-isometric diversity of left-orderable and perfect groups. 
\begin{prop}\label{IP:perfect} There are continuum many finitely generated left-orderable and perfect groups up to quasi-isometry. 
\end{prop}
 The groups in Proposition \ref{IP:perfect} are realised as small cancellation groups, see Section \ref{S:non-loc-ind}. Thus, Proposition \ref{IP:perfect} implies Proposition \ref{IP:qi-diversity}.  As mentioned in a personal communication by Arman Darbinyan, in a paper under preparation he is going to show quasi-isometric diversity of the class of left-orderable (in fact, bi-orderable) residually finite solvable groups.

\begin{rem} 
The results of this paper appear to be the first applications of small cancellation theory to the study of left-orderable groups. As all groups constructed in this article split as an amalgamated product or an HNN-extension, it would be interesting to construct left-orderable small cancellation groups that do not split. Examples of small cancellation groups that do not split were given in \cite{Pride_some_1983}, see also \cite[Lemma 5.1 and Remark 5.2]{Jankiewicz_cubulating_2022}. These groups are not left-orderable. In fact, their presentations contain relators contradicting left-orderability, see Example \ref{E:non-lo} below. We also note that random groups in the density model do not split \cite{dahmani_random_2011}.
\end{rem}

\subsection*{Acknowledgements}
The author thanks Arman Darbinyan and Yash Lodha for valuable discussions and encouragement, and Yago Antolin, Goulnara Arzhantseva, Arman Darbinyan, Thomas Delzant and Ashot Minasyan for useful comments on a previous version of this work. This work was supported by the Austrian Science Fund (FWF) project P 35079-N and the Mathematisches Forschungsinstitut Oberwolfach under its Research in Pairs program in year 2022 (Project ID 2212p). 

\section{Left-ordered groups}

\subsection{Lexicographic orderings of the free group}
 
Let $\mathcal{A}=\{a_1,a_2,a_3,\ldots, a_n\}$ be an alphabet. Every word $t$ in $\mathcal A$ such that each letter appears exactly once gives rise to an explicit left-order $\leq_t$ on the free group on $\mathcal{A}$ \cite{sunic_explicit_2013}. 
We now discuss such orders.
 
Given a word $t$ as above, we proceed as follows. For $x$ and $y$ in $\mathcal{A}$ we set $x<_t y$ and $y^{-1}<_t x^{-1}$ if $x$ comes before $y$ in the word $t$ read from left to right. We also set $x^{-1} <_t y$ for all $x,y \in \mathcal A$. Next let $g$ be an element in the free group on $\mathcal{A}$ that is represented as a reduced word, i.e. in normal form. Let 
\begin{align} \label{E:definition-tau}
\tau_{t} (g) =2 \sum_{x<_t y}\left(\#_{yx^{-1}}(g) - \#_{y^{-1}x}(g)\right),
\end{align}
where, for $x,y\in \mathcal{A}$, $\#_{yx^{-1}}(g)$ is the number of occurrences of $yx^{-1}$ in $g$, and $\#_{y^{-1}x}(g)$ is the number of occurrences of $y^{-1}x$ in $g$ respectively. Let $\omega(g)$ be $1$ if the exponent of the last letter of $g$ is positive, $-1$ if the exponent of the last letter of $g$ is negative, and $0$ if $g$ is the trivial word. 

We let $\leq_t$ be the relation on the free group on $\mathcal{A}$ defined by 
\begin{align*}
g \leq_t h \hbox{ whenever } \tau_t(g^{-1}h) + \omega(g^{-1}h)\geqslant 0.
\end{align*}

\begin{thm}[Th\'eor\`eme 0.2 of \cite{sunic_explicit_2013}]\label{T:left-order-free-group} Let $\mathcal{A}$ be a finite alphabet, and let $t$ be a word in $\mathcal{A}$ such that each letter appears exactly once. Then $\leq_t$ is a left-order on the free group on $\mathcal{A}$.
\end{thm}

The following facts on $\tau$ are useful to compare the explicit orders on free groups.  
\begin{lem}\label{L:tau} Let $\varepsilon,\varepsilon_1,\varepsilon_2,\varepsilon_3 \in \{\pm 1\}$ and $x,y,z\in \mathcal A$. Then  
\begin{enumerate}
\item $ \tau(x^{\varepsilon}y^{\varepsilon}) = \tau(x^{\varepsilon}z^{\varepsilon}y^{\varepsilon}), $
 \item $\tau(x^{\varepsilon_1}z^{\varepsilon_2}y^{\varepsilon_3}) = \tau(x^{2\varepsilon_1}z^{\varepsilon_2}y^{\varepsilon_3}) = \tau(x^{\varepsilon_1}z^{2\varepsilon_2}y^{\varepsilon_3}) = \tau(x^{\varepsilon_1}z^{\varepsilon_2}y^{2\varepsilon_3}),  $
\item $ \tau(x^{\varepsilon}y^{- \varepsilon}) = \tau(x^{\varepsilon}z^{\varepsilon}y^{-\varepsilon}), \hbox{ whenever  $x<_ty$ and $z<_ty$ or $x>_ty$ and $z>_ty$}, $
 \item $\tau(x^{\varepsilon}y^{-\varepsilon})= \tau(x^{\varepsilon}z^{-\varepsilon}y^{-\varepsilon}), \hbox{ whenever $x<_tz$ and $x<_ty$ or $x>_tz$ and $x>_ty$}  . $
\end{enumerate} 
\end{lem}

\begin{lem}\label{L:tau-2} Let $w_1$ and $w_2$ be elements of the free group such that $w_1w_2$ is reduced. Let $t_1$ be the last letter of $w_1$ and let $s_2$ be the first letter of $w_2$. Then 
$$ \tau(w_1w_2)=\tau(w_1) + \tau(w_2) +\tau(t_1s_2)=\tau(w_1s_2)+\tau(w_2)=\tau(w_1)+\tau(t_1w_2).$$
\end{lem}

\subsection{Explicit orderings on free products}\label{S:free-product}

If $G_1$ and $G_2$ are left-ordered, then the free product $G_1*G_2$ is left-orderable. We work with explicit left-orders on the free product given by \cite[Definition 4.1, Proposition 4.2]{warren_orders_2020}. These are given as follows. We represent every element $g$ of the free product $G_1*G_2$ in normal form, that is, 
$g=g_{i_1}g_{i_2}\cdots g_{i_m},
$ where $g_{i_j}$ is an element of $G_1$ or $G_2$ and two consecutive $g_{i_j}$ are in distinct groups. Each $g_{i_j}$ is refereed to as of a syllable of $g$; and a syllable is positive if it is positive in the order on $G_1$, or $G_2$ respectively. Let $\overline{\tau}$ be the function on the free product defined by 
\begin{align*}
\overline{\tau}(g)=\# \hbox{(positive syllables in $g$)} -& \# \hbox{(negative syllables in $g$)} \\
 & + \# \hbox{(index jumps in $g$)} - \# \hbox{(index drops in $g$)}, 
\end{align*}
where an index jump occurs at $j$ if $g_{i_j}\in G_1$ and $j<m$ (then $g_{i_{j+1}}\in G_2$), and an index drop occurs at $j$ if $g_{i_j}\in G_2$ and $j<m$ (then $g_{i_{j+1}}\in G_1$).
We let $\overline{\leq}$ be the relation on $G_1*G_2$ given by 
\begin{align*}
 g \overline{\leq} h \hbox{ whenever } \overline{\tau}(g^{-1}h) \geqslant 0.
 \end{align*}
\begin{thm}[Proposition 4.2 of \cite{warren_orders_2020}]\label{T:left-order-free-product} Let $G_1$ and $G_2$ be left-ordered groups. 
Then $\overline{\leq}$ is a left-order on $G_1*G_2$ that extends the given left-orders on $G_1$ and $G_2$.  
\end{thm}

\subsection{Combination theorems}
 Let $G$ be a left-ordered group and let $\leq$ be  a left-order on $G$. 
 If $g\in G$, we define the \emph{conjugate order} $\leq^g$ as follows: 
 $$ 1 \leq^g h \hbox{ if } 1 \leq ghg^{-1}, \hbox{for every $h\in G$}.$$
Let $\mathcal L$ be a set of left-orders on $G$. The set $\mathcal L$ is \emph{normal} if it is closed under conjugation by the elements of $G$, that is, for every $g\in G$ and order $\leq \in \mathcal L$, the conjugate order $\leq^g$ is in $\mathcal L$. 
The \emph{normal closure} of $\mathcal L$ is the set of left-orders that contains all conjugates of the left-orders in $\mathcal L$.

Let $G_1$ and $G_2$ be left-ordered groups, with subgroups $H_1\leq G_1$ and $H_2\leq G_2$ isomorphic to a group $H$. Let us denote the respective isomorphisms by $\varphi_i:H_i \to H$. Let $\mathcal L _1$ and $\mathcal L _2$ be a family of left-orders on $G_1$, and $G_2$ respectively. 
The isomorphism $\varphi:=\varphi_2^{-1}\varphi_1$ is \emph{compatible} for the pair $(\mathcal L_1,\mathcal L_2)$ if and only if 
for all $\leq_1 \in \mathcal L_1$  there is  $\leq_2 \in \mathcal L _2 $ such that  for all $h_1\in H_1$ we have 
\begin{align*}1\leq_1 h_1 \hbox{ implies that } 1\leq_2 \varphi(h_1).
\end{align*}
If families of left-orders $\mathcal L_1$ and $\mathcal L_2$ are fixed, if the isomorphism $\varphi$ is compatible for $(\mathcal L_1,\mathcal L_2)$ and if its inverse $\varphi^{-1}$ is compatible for $(\mathcal L_2,\mathcal L_1)$, then we say that $\varphi$ is  \emph{compatible} for short. 

We use the following theorems to produce negatively curved left-ordered groups. The first is about left-orderability of amalgams, the second of HNN-extensions. 

\begin{thm}[Theorem A of \cite{bludov_word_2009}]\label{T:combination-amalgam} Let $G_1$ and $G_2$ be left-ordered groups. 
 Let  $H_1\leq G_1$ and $H_2\leq G_2$ be isomorphic to a group $H$ via isomorphisms $\varphi_i:H_i \to H$. Then $G_1*_H G_2$ is left-orderable if, and only if, there are normal families of left-orders $\mathcal L _1$ and $\mathcal L _2$  on $G_1$, respectively $G_2$, such that $\varphi=\varphi_2^{-1}\varphi_1:H_1\to H_2$ is compatible. 
\end{thm}
 
\begin{thm}[Theorem B of \cite{bludov_word_2009}]\label{T:combination-HNN}
Let $G$ be a left-ordered group. 
 Let  $H_1\leq G$ and $H_2\leq G$ be isomorphic to a group $H$ via isomorphisms $\varphi_i:H_i \to H$. Then $G*_H$ is left-orderable if and only if there is a normal family of left-orders $\mathcal L$ on $G$ such that $\varphi=\varphi_2^{-1}\varphi_1:H_1\to H_2$ is compatible. 
\end{thm}
 
\section{Small cancellation groups}

Let $\mathcal{A}$ be an alphabet and let us denote by $|.|$ the word length of the elements in the free group over $\mathcal{A}$. 
If $g$ and $h$ are elements of the free group on $\mathcal{A}$, we denote its Gromov product by 
$(g,h)=\frac{1}{2}(|g|+|h|-|g^{-1}h|).$
Let $w_1,w_2,\ldots, w_i, \ldots$ be elements of the free group on $\mathcal{A}$ represented as words in reduced and cyclically reduced form. The words $w_1,w_2,\ldots, w_i, \ldots$, or the words $w_i$ for short, are \emph{small cancellation words} if $|w_i|>6$, for all $i>0$, and if 
$$ (v_1,v_2)<\frac{1}{6} \min\{|v_1|,|v_2|\},$$
for all words $v_1,v_2$ that are distinct reduced and cyclically reduced cyclic conjugates of the words $w_i$ or its inverses. 

If $G$ is a group given by a presentation $G=\langle \mathcal{A} \mid w_1,w_2,\ldots, w_i, \ldots\rangle$ and the words $w_i$ are small cancellation words, we say that $G$ is a small cancellation group. To be more precise, such groups are often refereed to as of classical  $C'(1/6)$-small cancellation groups. Classical results on small cancellation groups include their torsion-freeness if the relators $w_i$ do not contain any proper power, and,  under the additional assumption that the presentation of $G$ is finite, their hyperbolicity.

We frequently use the following fact. 
\begin{prop}\label{P:reduced-free} Let $w_1,w_2,\ldots, w_i, \ldots$ be elements of the free group on $\mathcal{A}$. If the words $w_i$ are small cancellation words, then $w_1,w_2,\ldots, w_i, \ldots$ freely generate a free subgroup in the free group on $\mathcal{A}$. 
\end{prop}

We need some more classical notation of small cancellation theory. Let $v_1$ and $v_2$ be two distinct elements of the free group on $\mathcal{A}$. A \emph{piece} of $v_1$ and $v_2$ is a common prefix $p$ of these two words, that is, $v_1=pv_1'$ and $v_2=pv_2'$. 
 If words $w_1,w_2,\ldots, w_i, \ldots$ are given, we refer to a \emph{piece} as of a piece of any two words $v_{1}$ and $v_{2}$ among all  distinct words that are reduced and cyclically reduced cyclic conjugates of the words $w_i$ or its inverses. 

\begin{rem} The words $w_i$ are small cancellation words, if, and only if, for every piece $p$ of any two words $v_{1}$ and $v_{2}$ among all  distinct words that are reduced and cyclically reduced cyclic conjugates of the words $w_i$ or its inverses, we have $|p|<\frac{1}{6}\min\{|v_1|,|v_2|\}$.
\end{rem}

We say that an element $g$ in the free group on $\mathcal{A}$ is \emph{$p$-reduced with $w_i$} if 
$$(g^{-1},w_i)\leqslant \max\{|p|\mid \hbox{$p$ is a piece and a prefix of $w_i$}\}.$$
Finally, if $p_i$ is a piece and a prefix of $w_i$, and if $p_i^{-1}$ is a suffix of $g$, we say that $g$ and $w_i$ are \emph{$p_i$-reduced}. 

Next we discuss two examples of small cancellation groups.  
Our first example shows that every torsion-free small cancellation construction can be made non-left-orderable. 

\begin{ex}\label{E:non-lo} Let $G$ be a group generated by $x_1,\ldots,x_n$. Let $r_1(x_1,x_2)$ be a positive word, i.e. the exponent of every letter $x_1$, $x_2$, is positive, and let $r_2(x_1,x_2)$ be a word such that every occurrence of $x_1$ has positive exponent, and every occurrence of $x_2$ has negative exponent. In addition, we assume that $r_1$ and $r_2$ are no proper powers. If $r_1$ and $r_2$ are relations in $G$, then $G$ is not left-orderable, as the subgroup generated by $x_1$ and $x_2$ is not left-orderable, see \cite[Theorem 1.48]{clay_ordered_2016}. Taking $r_1$ and $r_2$ appropriate small cancellation words, we observe that every small cancellation construction can be adapted to produce non-left-orderable groups by adding $r_1$ and $r_2$ to the presentation of the group. As $r_1$ and $r_2$ are no proper powers, $r_1$ and $r_2$ do not produce any new torsion elements. 
\end{ex}

The next example is a starting point for this study, and a first application of small cancellation theory to the study of left-orderable groups. 

\begin{ex}\label{E:Bowditch} Let $\beta_i(x,y)=x^iyx^iy^2x^iy^3\cdots x^iy^{20}$. 
Then $\{\beta_i \mid i>20\}$ is a set of small cancellation words whose length grows linearly in $i$. Let $\mathcal{N}:=\{2^{2^n}\mid n\in \mathbb{N} \hbox{ and $n>20$}\}$ and let $I\subset \mathcal N$. 
We denote by $B_I$ the group defined by  
$$
B_I=\langle a_1,a_2,b_1,b_2 \mid \beta_i(a_1,a_2)=\beta_i(b_1,b_2)\hbox{, for all $i \in I$} \rangle.
$$
We denote by $G_1$ the free group on $\{a_1,a_2\}$ and by $G_2$ the free group on $\{b_1,b_2\}$. Moreover, we let $H_I$  and $K_I$ be the free subgroup freely generated by the words $\beta_i(a_1,a_2)$, for all $i\in I$, in $G_1$, and $\beta_i(b_1,b_2)$, for all $i\in I$, in $G_2$, respectively. We let $\varphi:H_I\to K_I$ be the isomorphism sending $\beta_i(a_1,a_2)$ to $\beta_i(b_1,b_2)$ and note that 
$B_I=G_1*_{H_I}G_2$. By Theorem~\ref{T:combination-amalgam}, see \cite[Corollary 37]{bergman_ordering_1990} and \cite[Corollary 5.10]{bludov_word_2009}, the groups $B_I$ are left-orderable for all $I\subset \mathcal N$. As the groups $B_I$ are small cancellation groups, by the arguments of \cite{bowditch_continuously_1998}, see also \cite[Example 3.2]{minasyan_quasi_2021}, there are continuum many groups $B_I$ up to quasi-isometry. 
\end{ex}

We record the content of Example \ref{E:Bowditch} in the following proposition. 

\begin{prop}[{Corollary 37 of \cite{bergman_ordering_1990} and \cite{bowditch_continuously_1998}}] \label{P:Bowditch} There are continuum many left-orderable (small cancellation) groups up to quasi-isometry. 
\end{prop}
Combining this with \cite{rivas_left-orderings_2012} yields the following. 

\begin{cor}[{Corollary 37 of \cite{bergman_ordering_1990}, \cite{bowditch_continuously_1998}, \cite{rivas_left-orderings_2012}}] \label{C:Bowditch} There are continuum many (small cancellation) groups up to quasi-isometry whose space of left-orders is a Cantor set. 
\end{cor}
\begin{proof} We use the notation of Example \ref{E:Bowditch}. Let $I\subset \mathcal N$ and let $R_I=B_I*\mathbb{Z}$. By \cite{rivas_left-orderings_2012}, the space of left-orders of $R_I$ is a Cantor set. The group $R_I$ still is a small cancellation group and its relators are equal to the relators of $B_I$. In fact, Bowditch's taut length spectra of $R_I$ and $B_I$ equal, hence, the claim. 
\end{proof}

\section{Groups that are not locally indicable}\label{S:non-loc-ind}

We build left-orderable perfect groups $P$ using the combination theorem  for amalgams (Theorem 2.4). 
 Let $w_1(x,y),\ldots,w_4(x,y),v_1(x,y),\ldots,v_4(x,y)$ be small cancellation words such that, for all $1\leqslant i \leqslant 4$, 
\begin{enumerate}
\item the sum of the exponents of $x$ in $w_1(x,y),w_2(x,y)$,$v_1(x,y)$ and $v_2(x,y)$ is $-1$, and $0$ in all other words,
\item the sum of the exponents of $y$ in $w_3(x,y),w_4(x,y)$,$v_3(x,y)$ and $v_4(x,y)$ is $-1$, and $0$ in all other words.
\end{enumerate}
Moreover we assume that that all words start and end with $x$ or $y$. 

Let $G_1=<a,b>$ and $G_2=<c,d>$ be free groups. Let 
\[
\begin{array}{lll}
h_1=w_1(a,b)\; a^3 \; v_1(a,b) & & k_1= w_1(c,d)\;c^2\;v_1(c,d) \\
h_2=w_2(a,b)\;a^2\;v_2(a,b) & & k_2=w_2(c,d)\;c^3\;v_2(c,d) \\
h_3= w_3(a,b)\;b^3\;v_3(a,b) & & k_3= w_3(c,d)\;d^2\;v_3(c,d) \\
h_4=w_4(a,b)\;b^2\;v_4(a,b) & & k_4 = w_4(c,d)\;d^3\;v_4(c,d) \\ 
\end{array}
\]

Let $H=\langle h_1,\ldots, h_4\rangle \leq G_1$ and $K=\langle k_1,\ldots, k_4\rangle \leq G_2$. Note that $H$ and $K$ are free groups on the respective four generators (this follows from small cancellation assumption). We let $\varphi:H\to K$ be the isomorphism that sends $h_i$ to $k_i$, for all $i$. 
Let 
$$P=G_1*_{H}G_2$$
be the amalgamation over $H$ with respect to $\varphi$. We note that $P$ is a small cancellation group. 

\begin{prop}\label{P:perfect-lo}
The group $P$ is left-orderable and perfect.
\end{prop}

We now prove this proposition. We start with the following. 
\begin{lem}\label{L:perfect} The group $P$ is perfect. 
\end{lem} 
\begin{proof}
The sum of the exponents of $b,c$ and $d$ in the relation $h_1k_1^{-1}$ is zero, while the exponent sum of $a$ is $1$. This implies that the image of $a$ in the abelianisation of $P$ is trivial. Similarly, the image of the generators $c,b$ and $d$ is trivial in the abelianisation of $P$. Thus the abelianisation of $P$ is trivial, hence, $P$ is perfect. 
\end{proof}

To prove that $P$ is left-orderable we let $\leq_1$ be the left-order on $G_1$ defined by $o_1=ab$, and let $\mathcal{L}_1$ be its normal closure. Analogously, we let $\leq_2$ be the left-order on $G_2$ defined by $o_2=cd$, and let $\mathcal{L}_2$ be its normal closure. We want to apply Theorem \ref{T:combination-amalgam}. To this end we prove that $\varphi$ is compatible.  
 
We first fix some notation. For all $1\leqslant i \leqslant 4$, 
 we let
  $$
  \hbox{$h_{4+i}=h_i^{-1}$ and $k_{4+i}=k_i^{-1}$.}
  $$
We now fix words $w_i$ and $v_i$, letters $x_i$ and $y_i$, and numbers $n_i$ and $m_i$  such that 
$$h_i=w_{i}(a,b)x_i^{n_i}v_{i}(a,b) \hbox{ and } k_i=w_{i}(c,d)y_i^{m_i}v_{i}(c,d).$$
To this point, we let $w_{4+i}=v_i^{-1}$ and $v_{4+i}=w_i^{-1}$. 
 We let $x_1=x_2=a$, $x_3=x_4=b$, $x_5=x_6=a^{-1}$, $x_7=x_8=b^{-1}$, and $y_1=y_2=c$, $y_3=y_4=d$, $y_5=y_6=c^{-1}$, $y_7=y_8=d^{-1}$. Moreover, we let $n_1=n_3=n_5=n_7=3$, $n_2=n_4=n_6=n_8=2$, $m_1=m_3=m_5=m_7=2$ and $m_2=m_4=m_6=m_8=3$.  

 Let us also denote by $t_i$ the last letter of $w_i$ and by $s_i$ the first letter of $v_i$. Finally, let 
 $\tau_1=\tau_{o_1}$ and $\tau_2=\tau_{o_2}$ be the functions used to define $\leq_1$ and $\leq_2$ respectively, see \eqref{E:definition-tau}.

\begin{rem}\label{R:perfect-1}
 As $w_i$ and $v_j$ are small cancellation words, $v_jw_i$ is either trivial or of length at least $12$. In particular, $v_jw_i$ starts with $s_j$ and ends with $t_i$. Moreover, by definition, 
 $$\tau_1(v_jw_i(a,b))=\tau_2(v_jw_i(c,d)).$$
\end{rem}
\begin{rem}\label{R:perfect-2} Let $u_{i1}(x,y)$ and $u_{i2}(x,y)$ be words such that $u_{i1}(a,b)x_i$ and $x_iu_{i2}(a,b)$ is reduced. Then Lemma \ref{L:tau} gives that 
\begin{align*}
\tau_1(u_{i1}(a,b)x_i^{n_i}u_{i2}(a,b)) & = \tau_2(u_{i1}(c,d)y_i^{m_i}u_{i2}(c,d)), \\
\tau_1(u_{i1}(a,b)x_i^{n_i-1}u_{i2}(a,b))& = \tau_2(u_{i1}(c,d)y_i^{m_i-1}u_{i2}(c,d)), \\
\tau_1(u_{i1}(a,b)x_iu_{i2}(a,b))& = \tau_2(u_{i1}(c,d)y_iu_{i2}(c,d)).
\end{align*} 
This is the case, in particular, if $u_{i1}=t_i$ and ${u_{i2}}=s_i$. 
\end{rem}

Remark \ref{R:perfect-2} and Lemma \ref{L:tau-2} imply that $\tau_1(h_i)=\tau_2(k_i)$, for all $i>0$. In particular, 
$1\leq_1 h_i$ if and only if $1\leq_2 k_i$. In more generality, we have the following.

\begin{lem}\label{L:perfect-LO-1} Let $g\in H$ and let $f=\varphi(g)$. Then 
$$1\leq_1^g h \hbox{ if and only if } 1\leq_2^f\varphi(h), \hbox{ for all $h\in H$.}$$
\end{lem}
\begin{proof}
Let $h\in H$. As $g\in H$ it is sufficient to prove that $1\leq_1 h$ if and only if $1\leq_2 \varphi(h)$. 
We expand $h=h_{i1}h_{i_2}\cdots h_{il}$ as 
$$h=w_{i1}(a,b)x_{i_1}^{n_{i_1}}v_{i1}w_{i2}(a,b)x_{i_2}^{n_{i_2}}v_{i2}(a,b)\cdots w_{il}(a,b)x_{i_l}^{n_{i_l}}v_{il}(a,b).$$
Then 
$$\varphi(h)=w_{i_1}(c,d)y_{i_1}^{m_{i_1}}v_{i_1}(c,d)w_{i_2}(c,d)y_{i_2}^{m_{i_2}}v_{i_2}(c,d)\cdots w_{i_l}(c,d)y_{i_l}^{m_{i_l}}v_{i_l}(c,d).$$
By Lemma \ref{L:tau-2}, $\tau_1(h)$  is equal to 
\begin{align*}
\tau_1(w_{i_1}(a,b))+\sum_{j=1}^{l-1} \tau_1(v_{i_j}w_{i_{j+1}}(a,b))& + \tau_{1}(v_{i_l}(a,b))+ \sum_{j=1}^l \tau_1(t_{i_j}(a,b)x_{i_j}^{n_{i_j}}s_{i_j}(a,b)),
  \end{align*}
  and $\tau_2(\varphi(h))$ is equal to
  \begin{align*}
 \tau_2(w_{i_1}(c,d))+\sum_{j=1}^{l-1} \tau_2(v_{i_j}w_{i_{j+1}}(c,d)) & +\tau_2(v_{i_l}(c,d))+ \sum_{j=1}^l \tau_2(t_{i_j}(c,d)y_{i_j}^{m_{i_j}}s_{i_j}(c,d)).
\end{align*}
By Remarks \ref{R:perfect-1} and \ref{R:perfect-2} these two sums equal, hence, $\tau_1(h)=\tau_2(\varphi(h))$. We observe that $\omega(h)=\omega(\varphi(h))$, which yields the assertion of the lemma. 
\end{proof}

We complement the previous lemma by the following statements. 

\begin{lem}\label{L:perfect-LO-2} Let $h=h_{i_1}\cdots h_{i_l}$. Let $g_1,g_3\in G_1$ and $g_2,g_4\in G_2$ such that $g_1$,$g_3$, $g_2$  and $g_4$ are $p$-reduced with $h_{i_1}$, ${h_{i_l}}^{-1}$, $k_{i_1}$ and ${k_{i_l}}^{-1}$ respectively. If 
$$ \tau_1(g_1w_{i_1}(a,b))=\tau_2(g_2w_{i_1}(c,d)) \hbox{ and }  \tau_1(g_3v_{i_l}^{-1}(a,b))=\tau_2(g_4v_{i_l}^{-1}(c,d)),$$ 
then 
$$\tau_1(g_1hg_3^{-1})=\tau_2(g_2\varphi(h)g_4^{{-1}}).$$
\end{lem}
\begin{proof}
If $g$ is $p$-reduced with $h_i$, the words $gw_i(a,b)$  are of length at least $6$ and terminate with $t_i(a,b)$. If $g$ is $p$-reduced with $h_i^{-1}$, the words $v_i(a,b)g^{-1}$ are of length at least $6$ and start with $s_i(a,b)$. 
We expand $g_1hg_3^{-1}$ as 
$$g_1w_{i_1}(a,b)x_i^{n_i}v_i(a,b)h_{i_2}\cdots h_{i_{l-1}}w_{i_l}(a,b)x_{i_l}^{n_{i_l}}v_{i_l}(a,b)g_3^{-1}.$$
By Lemma \ref{L:tau-2},  $ \tau_1(g_1hg_3^{-1})$ is equal to
$$\tau_1(g_1w_{i_1}(a,b)) + \tau_1(t_{i_1}x_i^{n_i}v_i(a,b)h_{i_2}\cdots h_{i_{l-1}}w_{i_l}(a,b)x_{i_l}^{n_{i_l}}s_{i_l})+\tau_1(v_{i_l}(a,b)g_3^{-1}).$$
Analogously,  $\tau_2(g_2\varphi(h){g_4}^{-1})$ is equal to
$$\tau_2(g_2w_{i_1}(c,d)) + \tau_2(t_{i_1}(c,d)y_i^{m_i}v_i(c,d)k_{i_2}\cdots k_{i_{l-1}}w_{i_l}(c,d)y_{i_l}^{m_{i_l}}s_{i_l})+\tau_2(v_{i_l}(c,d){g_4}^{-1}).$$
As in the proof of Lemma \ref{L:perfect-LO-1}, we observe that the two middle terms in these sums equal. The other terms equal by assumption. This yields the claim. 
\end{proof}
If $g=g(a,b) \in G_1$, we denote by $g'$ the element $g(c,d)\in G_2$.
\begin{cor}\label{C:perfect-LO-2} If $g\in G_1$ is $p$-reduced with all $h_i$, then 
$$ 1\leq_1^g h \hbox{ if and only if } 1\leq_2^{g'} \varphi(h).$$
\end{cor}
\begin{proof}
We observe that
\begin{align*}
 \tau_1(gw_i(a,b))=\tau_2(g'w_i(c,d)) \hbox{ and } \tau_1(v_i(a,b)g^{-1})=\tau_2(v_i(c,d){g'}^{-1}).
\end{align*}
Also $\omega(ghg^{-1})=\omega(g'\varphi(h){g'}^{-1})$. The assertion now follows from Lemma \ref{L:perfect-LO-2}. 
\end{proof}
 Let $p_i$ denote the piece that is maximal with the property that it is a suffix of $v_i$. 
\begin{lem}\label{L:perfect-LO-3} Let $g\in G_1$ such that $(g^{-1},h_i)< |h_i|-|p_i|$, for all $1\leq i\leq 8$. 
Then there is $f\in G_2$ such that 
$ 1\leq_1^g h \hbox{ if and only if } 1\leq_2^f \varphi(h), \hbox{ for all $h\in H$.}$
\end{lem}
\begin{proof}  Let $g\in G_1$ such that $(g^ {-1},h_i)< |h_i|-|p_i|$, for all $1\leqslant i\leqslant 8$. Note that this implies that the reduced form of $gh_i$ ends with $rp_i$, where $r$ is a non-trivial word.  In particular, $gh_i$ is $p$-reduced with all $h_j\not=h_i^{-1}$.

If $g$ is $p$-reduced with all $h_i$, we are done by Corollary \ref{C:perfect-LO-2}. If $g$ is not $p$-reduced with $h_i$, it has to be $p$-reduced with all other $h_j$, $j\not=i$. There are 4 cases to consider. 
We fix $i$ and let $j$ such that $v_j=w_i^{-1}$ and $w_j=v_i^{-1}$. 

{\emph{Case 1:} $g=g_0u_j(a,b)$ and $gh_i=g_0u_j(a,b)w_i(a,b)x_i^{n_i}v_i(a,b)$, where $u_j$ is a suffix of $v_j$ that is not a piece.}
Then let $f=g_0'u_j(c,d)$ so that $fh_i=g_0'u_j(c,d)w_i(c,d)y_i^{m_i}v_i(c,d)$.

{\emph{Case 2:} $g=g_0x_i^{-1}v_j(a,b)$ and $gh_i=g_0x_i^{n_i-1}v_i(a,b)$.} \\ 
Then let $f=g_0'y_i^{-1}v_j(c,d)$, so that $fk_i=g_0'y_i^{m_i-1}v_i(c,d)$.

{\emph{Case 3:} $g=g_0x_i^{-n_i+1}v_j(a,b)$ and $gh_i=g_0x_iv_i(a,b)$.}
\\ Then let $f=g_0'y_i^{-m_i+1}v_j(c,d)$, so that $fk_i=g_0'y_iv_i(c,d)$.

{\emph{Case 4:}  $g=g_0u_{j}(a,b)x_i^{-n_i}v_j(a,b)$ and $gh_i=g_0u_j(a,b)v_i(a,b)$, where $u_j$ is a suffix of $w_j$.} 
Then let $f=g_0'u_{j}(c,d)y_i^{-m_i}v_j(c,d)$, so that $fk_i=g_0'u_j(c,d)v_i(c,d)$. 

We now prove the assertion of the lemma in Case 4. The proof in Cases 1-3 is similar.

We use the notation of Case 4. Note that $u_jv_i$ is usually not in reduced form, however, it equals to a reduced word $r_{q_1}\ldots r_{q_t}p_i$, where $t\geqslant 1$. Note that the reduced form of $u_jv_iw_i(a,b)$ starts with $r_{q_1}$ and ends with $t_i$. We first prove that 
$$\tau_1(gh_ig^{-1})=\tau_2(fk_if^{-1}).$$
Indeed, $gh_ig^{-1} = g_0u_jv_iw_i(a,b)x_i^{n_i}u_j^{-1}(a,b)g_0^{-1}$.
 Thus $\tau_1(gh_ig^{-1})$ equals
$$\tau_1(g_0r_{q_1})+\tau_1(u_jv_iw_i(a,b)) + \tau_1(t_ix_i^{n_i}s_i) + \tau_1(u_j^{-1}(a,b)g_0^{-1}).$$
Similarly, 
$\tau_2(fk_if^{-1})=\tau_2(g_0'r_{q_1}')+\tau_2(u_jv_iw_i(c,d)) + \tau_2(t_i'y_i^{m_i}s_i') + \tau_2(v_i(c,d)g_0'^{-1}).$
We now readily observe that these two sums equal, see Remark \ref{R:perfect-2}.

Next let $h=h_{i_1}\ldots h_{i_l}$, for $n>1$ or $n=1$ and $h_{i_1}\not=h_i^{\pm 1}$. As $g$ is $p$-reduced with all $h_s\not = h_i$, and $gh_i$ is $p$-reduced with all $h_s\not=h_i^{-1}$, $1\leqslant s\leqslant 8$, it is sufficient to verify the assumptions of Lemma \ref{L:perfect-LO-2}. To do so, one needs to verify that 
$$\tau_1({g}w_s(a,b))=\tau_2(fw_s(c,d)) \hbox{ and } \tau_1(gh_iw_s(a,b))=\tau_2(fk_iw_s(c,d)).$$
This is analogous to the above arguments to prove that $\tau_1(gh_ig^{-1})=\tau_2(fk_if^{-1})$.
Thus, by Lemma \ref{L:perfect-LO-2}, 
$\tau_1(ghg^{-1})=\tau_2(f\varphi(h)f^{-1}).$ 

Finally, we note that $\omega(ghg^{-1})=\omega(f\varphi(h)f^{-1})$ by definition. This completes the proof of the lemma.
\end{proof}

\begin{lem}\label{L:perfect-LO-4} Let $g\in G_1$ such that $(g^{-1},h_i)\geqslant |h_i|-|p_i|$, for some $1\leq i\leq 8$. 
Then there is $f\in G_2$ such that 
$ 1\leq_1^g h \hbox{ if and only if } 1\leq_2^f \varphi(h), \hbox{ for all $h\in H$.}$
\end{lem}

\begin{proof} Let $h_i$ and  $g\in G_1$ such that $(g^{-1},h_i)\geqslant |h_i|-|p_i|$. Below, we recursively define $g_0\in G_1$  such that $g=g_0h_0$, for some $h_0\in H$, and such that $(g_0^{-1},h_j)< |h_j| - |p_j|$, for all $h_j$, $1\leqslant j \leqslant 8$. Note that  $1\leq^g h$ if and only if $1 \leq^{g_0} h_0hh_0^{-1}$. As $h_0Hh_0^{-1}=H$, the assertion then follows from Lemma \ref{L:perfect-LO-3}. 

We now construct $g_0$. 
Fix $g_1$ such that $g=g_1h_i^{-1}$. Now suppose that $g_i$ is given. If $(g_i^{-1},h_{j_i})\geqslant |h_{j_i}| - |p_{j_i}|$, for some $1\leqslant j_i\leqslant 8$, then we let $g_{i+1}$ such that $g_i=g_{i+1}h_{j_{i}}^{-1}$ and $g=g_{i+1}h_{j_i}^{-1}h_{j_{i-1}}^{-1}\cdots h_i^{-1}$.
Otherwise, we let $g_0=g_i$. This process stops because the length of $g$ is finite.  
\end{proof}

\begin{proof}[Proof of Proposition \ref{P:perfect-lo}]
The group $P$ is perfect by Lemma \ref{L:perfect}. By Lemmas \ref{L:perfect-LO-3} and \ref{L:perfect-LO-4}, $\varphi$ is compatible for  $(\mathcal{L}_1,\mathcal{L}_2)$. The proof that $\varphi^{-1}$ is compatible for $(\mathcal{L}_2,\mathcal{L}_1)$ is analogous. Thus $\varphi$ is compatible, and $P$ is left-orderable by Theorem \ref{T:combination-amalgam}.
\end{proof}

We now combine the presentation of $P$ with Example \ref{E:Bowditch}.

\begin{proof}[Proof of Propositions \ref{IP:qi-diversity} and  \ref{IP:perfect}] We use the notation of Example \ref{E:Bowditch}. For $i>20$ we let $h_{2i}=\beta_i(a,b)$ and $k_{2i}=\beta_i(c,d)$, and $h_{2i+1}=h_{2i}^{-1}$ and $k_{2i+1}=k_{2i}^{-1}$. The words $h_1,\ldots,h_4$ and $k_1,\ldots,k_4$ are chosen as above, where  $w_1$, $w_2$, $w_3$, $w_4$ and $v_1$, $v_2$, $v_3$ and $v_4$ are such that the words $h_1,\ldots,h_4$, $k_1,\ldots,k_4$ and all the words $h_{2i}$ and $k_{2i}$, for all $i>0$, are small cancellation words. We let $I\subset \mathcal{N}$, and denote by $H_I$ and $K_I$ the free subgroup freely generated by $h_1,h_2,h_3,h_4$ and the elements $h_{2i}$, for all $i\in I$, the free subgroup freely generated by $k_1,k_2,k_3,k_4$ and the elements $k_{2i}$, for all $i\in I$, respectively.  
Then we let 
$$
P_I=G_1*_{H_I}G_2.
$$
The proof of Lemma \ref{L:perfect} implies that $P_I$ is perfect. The proof that $P_I$ is left-ordered is as follows: for every $i>8$, we split $h_i=w_iv_i$ as a product, where $w_i$ and $v_i$ are words such that $|v_i|-1\leqslant |w_i|\leqslant |v_i|+1$. We also let $x_i=y_i$ be the empty word, and $n_i=m_i=1$, so that $h_i=w_i(a,b)x_i^{n_i}v_i(a,b)$ and $k_i=w_i(c,d)y_i^{m_i}v_i(c,d)$. Using this notation, we note that Lemmas \ref{L:perfect-LO-1}, \ref{L:perfect-LO-2} and Corollary \ref{C:perfect-LO-2} hold for $P_I$. In fact, the proofs are without any change.  

Now let $g\in G_1$ such that $(g^{-1},h_i)< |h_i|-|p_i|$, for all $i\geqslant 1$. Then there is $f\in G_2$ such that 
$ 1\leq_1^g h \hbox{ if and only if } 1\leq_2^f \varphi(h), \hbox{ for all $h\in H$,}$ cf. Lemma \ref{L:perfect-LO-3}.

Indeed, if $g$ is $p$-reduced with all $h_i$, $i\geqslant 1$, the claim is by Corollary \ref{C:perfect-LO-2}. If $g$ is not $p$-reduced with $h_i$, for some $i\geqslant 1$, then $g$ is $p$-reduced with all other $h_j$, $j\not=i$. Thus, there are two cases: if $g$ is not $p$-reduced with some $h_i$, for $1\leqslant i \leqslant 8$, we choose $f$ as given by Lemma \ref{L:perfect-LO-3}. With the above notation, the proof of this lemma then applies. We thus conclude the claim in this case. Otherwise $g$ is not $p$-reduced with some $h_i$, $i>8$. Then we let $f=g'=g(c,d)$. With this choice of $f$ the definitions imply that $\tau_1(gh_ig^{-1})=\tau_2(fk_if^{-1})$, and moreover that $\tau_1({g}w_s(a,b))=\tau_2(fw_s(c,d))$ and $\tau_1(gh_iw_s(a,b))=\tau_2(fk_iw_s(c,d))$, for all $s\geqslant 1$. Lemma \ref{L:perfect-LO-2} then implies the claim. 

We conclude that Lemma \ref{L:perfect-LO-4} holds for all $i\geqslant 1$. Thus the groups $P_I$ are left-orderable for all $I\subset \mathcal N$ by Theorem \ref{T:combination-amalgam}. As the groups $P_I$ are small cancellation groups, we conclude, using  \cite{bowditch_continuously_1998}, see also \cite[Example 3.2]{minasyan_quasi_2021}, that there are continuum many such groups up to quasi-isometry. This yields Proposition \ref{IP:perfect}. The abelianisation of a locally indicable group is non-trivial, in fact, it always contains the infinite cyclic group. As perfect groups are the groups with a trivial abelianisation, we have proven Proposition \ref{IP:qi-diversity}.
\end{proof}

\begin{rem}\label{R:Cantor} Replacing the group $B_I$ by the group $P_I$ in the proof of Corollary \ref{C:Bowditch} yields continuum many quasi-isometry classes of small cancellation groups that are not locally indicable and whose space of left-orders is homeomorphic to a Cantor set, see Remark \ref{IR:Cantor}.
\end{rem}

\section{Rips construction}

In this section, we prove Theorem \ref{IT:Rips} and Proposition \ref{IP:not-locally-indicable}. 

\subsection{Rips construction via HNN-extensions}
Let $Q$ be a finitely generated group given by  
$$Q=\langle x_1,\ldots x_n \mid r_1,r_2,\ldots ,r_i,\ldots \rangle.$$
Let $F$ be  a free group given by 
$$F=<a_1,a_2,b_1,b_2,c,x_1,\ldots,x_n,d_1,d_2,d_3,y_1,\ldots y_n, e>.$$ Below we write $z_1=a_1$, $z_2=a_2$, $z_3=b_1$, $z_4=b_2$, $z_5=c$, $z_6=d_1$, $z_7=d_2$, $z_8=d_3$ and $z_9=e$. Moreover, $z_1'=b_1$, $z_2'=b_2$, $z_3'=a_1$, $z_4'=a_2$ and $z_i'=z_i$ for all $i>4$. Let  
\allowdisplaybreaks 
\begin{align}
h_{j_i}=w_{j_i}(a_1,a_2)\; x_i, & \quad k_{j_i}= w_{j_i}(b_1,b_2)\; x_i  \hbox{, where $j_i = 2i$}; \label{rips-relator-1}\\
h_{j_i}=w_{j_i}(a_1,a_2)\; x_i^{-1}, & \quad  k_{j_i}= w_{j_i}(b_1,b_2)\; x_i^{-1}  \hbox{, where $j_i = 2i+2n$}; \label{rips-relator-2}\\
h_{j_{i,l}}=w_{j_{i,l}}(a_1,a_2)x_i\; y_l \; x_i^{-1}v_{j_{i,l}}(a_1,a_2), & \quad  k_{j_{i,l}}=w_{j_{i,l}}(b_1,b_2)y_i\; d_2 \; y_i^{-1}v_{j_{i,l}}(b_1,b_2) \label{rips-relator-3}\\
& \hspace{1cm} \hbox{, where $j_{i,l} = 2i+4n+2(l-1)n$} \nonumber; \\
h_{j_{i,l}}=w_{j_{i,l}}(a_1,a_2)x_i^{-1}\; y_l \; x_iv_{j_{i,l}}(a_1,a_2), & \quad  k_{j_{i,l}}=w_{j_{i,l}}(b_1,b_2)y_i^ {-1}\; d_2 \; y_iv_{j_{i,l}}(b_1,b_2) \label{rips-relator-4}\\
&  \hbox{, where $j_{i,l} = 2i+4n+2n^ 2+2(l-1)n$} \nonumber;\\
h_{j_{i,l}}=w_{j_{i,l}}(a_1,a_2)x_i\; z_l \; x_i^{-1}v_{j_{i,l}}(a_1,a_2), & \quad  k_{j_{i,l}}=w_{j_{i,l}}(b_1,b_2)y_i\; z_l' \; y_i^ {-1}v_{j_{i,l}}(b_1,b_2) \label{rips-relator-5}\\
&  \hbox{, where $j_{i,l} = 2i+4n+4n^ 2+2(l-1)n$} \nonumber; \\
h_{j_{i,l}}=w_{j_{i,l}}(a_1,a_2)x_i^{-1}\; z_l \; x_iv_{j_{i,l}}(a_1,a_2), & \quad  k_{j_{i,l}}=w_{j_{i,l}}(b_1,b_2)y_i^ {-1}\; z_l' \; y_iv_{j_{i,l}}(b_1,b_2) \label{rips-relator-6}\\
&  \hbox{, where $j_{i,l} = 2i+22n+4n^ 2+2(l-1)n$} \nonumber; \\
h_{i}=w_{i}(a_1,a_2)\; r_{i}(x_1,\ldots,x_n) \;v_{i}(a_1,a_2), & \quad  k_{i}=w_{i}(b_1,b_2)\; r_{i}(y_1,\ldots,y_n) \; v_{i}(b_1,b_2) \label{rips-relator-7}\\
& \hspace{2cm} \hbox{, where $i = 2i+40n+4n^2$}. \nonumber
\end{align}
We choose $w_i=w_i(x,y)$ and $v_i=v_i(x,y)$  so that the words $h_{2i}$ and the words $k_{2i}$ are small cancellation words.

Let $H=\langle h_2,\ldots, h_{2i},\ldots \rangle \leq F$ and $K=\langle k_2,\ldots, k_{2i},\ldots \rangle \leq F$. Note that $H$ and $K$ are free groups on the respective generators, see Proposition \ref{P:reduced-free}. We let $\varphi:H\to K$ be the isomorphism that sends $h_{2i}$ to $k_{2i}$, for all $i$. 
Let 
$$G=F*_{H}=\langle F,q \mid h=q \varphi(h) q^{-1},\hbox{ for all $h\in H$} \rangle$$
be the HNN-extension over $H$ with respect to $\varphi$. We note that $G$ is a small cancellation group. Moreover, if $Q$ is given by a finite presentation, then $G$ is finitely presented and hyperbolic. 

\begin{prop}\label{P:Rips}
The group $Q$ is a quotient of $G$ by an $(n+10)$-generated kernel. 
\end{prop}
\begin{proof} The relators $h_i=qk_iq^{-1}$ where $1\leqslant i\leqslant 40n+4n^ 2$ (i.e. $h_i$ and $k_i$ from (5)-(10)) imply that the subgroup $N=\langle q,a_1,a_2,b_1,b_2,c,d_1,d_2,d_3,y_1,\ldots y_n, e\rangle $ is normal in $G$. The remaining relators (coming from (11)) show that the map defined by sending $x_i$ to $x_i$, for all $1\leqslant i \leqslant n$, and all other generators to the identity is a projection onto $Q$ with kernel $N$. 
\end{proof}

\subsection{Construction of left-orders}

In this section, we prove that $G$ is left-orderable. To this end we first define two explicit left-orders on $F$. We let 
\begin{align*}
 o_1& = a_1a_2\; b_1b_2\; c \; x_1x_2\cdots x_n \; d_1d_2d_3 \; y_1y_2\cdots y_n \; e, \\
  o_2& = b_1b_2\; a_1a_2\; c \; y_1y_2\cdots y_n \; d_1d_2d_3 \; x_1x_2\cdots x_n \; e
\end{align*}
 and denote the left-order defined by $o_1$ by $\leq_1$, and the left-order defined by $o_2$ by $\leq_2$ respectively, see Theorem \ref{T:left-order-free-group}. 
 Note that in $o_1$ and $o_2$ the blocks $a_1a_2$ and $b_1b_2$ are swapped, the blocks $ x_1x_2\cdots x_n$ and $ y_1y_2\cdots y_n$ are swapped, while the position of all other letters is fixed. Let $\phi:F\to F$ be the map defined by a permutation of the generators such that $\phi(o_1)=o_2$. If $g$ is a word in $F$, we let $g'=\phi(g)$.  For instance, $a_1'=b_1$, $x_i'=y_i$, and $c'=c$. 
 
  If it is not necessary to distinguish $\leq_1$ and $\leq_2$ we write $\leq$ to denote any of the two orders, and denote by $\leq'$ the other one. For instance $a_1\leq x_1$ is a true statement, as $a_1\leq_1 x_1$  and $a_1\leq _2 x_1$. Also, if, for example,  $x_1<y_1$, this implies that $\leq=\leq_1$. 
  
 Similarly, we let $\tau_1=\tau_{o_1}$ and $\tau_2=\tau_{o_2}$, see \eqref{E:definition-tau}, and write $\tau$ to denote any of the two functions, and denote by $\tau'$ the other one.
 
 Let $ A=\{a_1,a_2,b_1,b_2\}$, $X=\{x_1,\ldots,x_n\}$ and $Y=\{y_1,\ldots,y_n\}$. If it is not important to which of the letters we refer, we often write $A$ to represent a letter in ${A}$, $X$ for a letter in $X$ and $Y$ for a letter in $Y$.

\begin{rem}\label{R:rem-rips-1} We let $\varepsilon_1,\varepsilon_2,\varepsilon_3,\varepsilon_4\in \{\pm1\}$ and note that 
\begin{align}
& \tau(A^{\varepsilon_1}X^{\varepsilon_2}A^{\varepsilon_3})=\tau'({A}^{\varepsilon_1}X^{\varepsilon_2}{A}^{\varepsilon_3}), \\
& \tau(A^{\varepsilon_1}x_{i_1}^{\varepsilon_2}x_{i_2}^{\varepsilon_3}A^{\varepsilon_4})=\tau'({A}^{\varepsilon_1}x_{i_1}^{\varepsilon_2}x_{i_2}^{\varepsilon_3}A^{\varepsilon_4}), \hbox{ for all $x_{i_1}$, $x_{i_2}$  in $X$},\\
& \tau(Y^{\delta}X^{\varepsilon})=\tau'(d_2^{\delta}Y^{\varepsilon}), \label{E:rem-rips-1-4} \\
& \tau(A^{\varepsilon_1}X^{\varepsilon_2}Y^{\varepsilon_3}X^{-\varepsilon_2}A^{\varepsilon_4})=\tau'({A}^{\varepsilon_1}Y^{\varepsilon_2}d_2^{\varepsilon_3}Y^{-\varepsilon_2}{A}^{\varepsilon_4}). \label{E:rem-rips-1-3} 
\end{align}
Let us justify \eqref{E:rem-rips-1-4} and \eqref{E:rem-rips-1-3}: we have that $A < Y$ and $A<X$. Moreover, $X<Y$ if and only if $Y<'d_2$, and $X>Y$ if and only if $Y>'d_2$.  This yields the equalities. The other equalities are straight-forward consequences of the choice of the orders. 
\end{rem}

We fix some more notation. We let $h_{2i-1}$ be the inverse of $h_{2i}$, and let $w_{2i-1}=v_{2i}^{-1}$ and $v_{2i-1}=w_{2i}^{-1}$. If it is not strictly necessary to keep track of the index we sometimes omit it. 

To prove that $G$ is left-orderable, we want to apply the combination theorem for HNN-extensions, Theorem \ref{T:combination-HNN}. We let $\mathcal L$ be the normal closure of $\{\leq_1,\leq_2\}$ in $F$. Then we need to prove that $\varphi$ is compatible. We first prove that $\varphi$ is compatible with respect to $(\mathcal{L},\mathcal{L})$, and come to the compatibility of $\varphi^{-1}$ with respect to $(\mathcal{L},\mathcal{L})$ later. 

Remark \ref{R:rem-rips-1} combined with Lemma \ref{L:tau-2} implies that $\tau(h_i)=\tau'(k_i)$, for all $i>0$. This immediately implies that $1< h_i$ if, and only if, $1<'k_i$. Remark \ref{R:rem-rips-1} is only needed to deal with \eqref{rips-relator-1}, \eqref{rips-relator-2}, \eqref{rips-relator-3} and \eqref{rips-relator-4}; all the other cases are by definition. 
Similarly, we apply Remark \ref{R:rem-rips-1} and Lemma \ref{L:tau-2} to obtain the following analogue of Lemma \ref{L:perfect-LO-1}. 

\begin{lem}\label{L:Rips-LO-1} Let $g\in H$ and let $f=\varphi(g)$. Then 
$$1\leq^g h \hbox{ if and only if } 1{\leq'}^f\varphi(h), \hbox{ for all $h\in H$.}$$
\end{lem}

 The proof is analogous to the proof of Lemma \ref{L:perfect-LO-1}, where Remark \ref{R:rem-rips-1} is used in place of Remark \ref{R:perfect-2}, and we do not repeat it.  
 We now prove an analogue of Lemma \ref{L:perfect-LO-2}.
 We denote, for every piece $p_i$ that is a prefix  or suffix of $h_i$ by $\overline{p_i}$ the word obtained by replacing $a_1$ by $b_1$ and  $a_2$ by $b_2$, and vice versa, but where the letters $x_i$ are fixed. Note that such a piece $p_i$ never contains a letter of $Y$. 
 
\begin{lem}\label{L:Rips-LO-2} Let $h=h_{i_1}\cdots h_{i_l}$. Let $g_1,g_3$ and $g_2,g_4\in F$ such that $g_1$ and $h_{i_1}$ are  $p_{i_1}$-reduced, $g_2$ and $k_{i_1}$ are $\overline{p_{i_1}}$-reduced, $g_3$  and $h_{i_l}^{-1}$ are $p_{i_l}$-reduced and $g_4$ and  $k_{i_l}^{-1}$ are $\overline{p_{i_l}}$-reduced. If, for every $\varepsilon\in\{\pm 1\}$, $a\in A$ and $x\in X$, 
\begin{enumerate}
\item $ \tau(g_1p_{i_1}a^{\varepsilon})=\tau'(g_2{\overline{p_{i_1}}}{a'}^{\varepsilon})$ and $\tau(g_1x^{\varepsilon})=\tau'(g_2{x}^{\varepsilon})$, 
\item $ \tau(g_3p_{i_l}a^{\varepsilon})=\tau'(g_4{\overline{p_{i_l}}}{a'}^{\varepsilon}) $ { and }  $\tau(g_3x^{\varepsilon})=\tau'(g_4{x}^{\varepsilon})$, 
\end{enumerate}
 then 
$$\tau(g_1hg_3^{-1})=\tau'(g_2\varphi(h)g_4^{{-1}}).$$
\end{lem} 

\begin{proof} Suppose that $p_{i_1}^{-1}h_{i_1}$ starts with $a\in A$. Then we expand 
$$\tau(g_1hg_3^{-1})=\tau(g_1p_{i_1}a)+\tau(p_{i_1}^{-1}h_{i_1}\cdots h_{i_l}g_2^{-1})$$
and 
$$\tau'(g_2\varphi(h)g_4^{-1})=\tau'(g_2\overline{p_{i_1}}a')+\tau'(\overline{p_{i_1}}^{-1}k_{i_1}\cdots k_{i_l}g_4^{-1}).$$
If $p_{i_1}$ is trivial and $h_{i_1}$ starts with $x\in X$, then we expand 
$$\tau(g_1hg_3^{-1})=\tau(g_1x)+\tau(h_{i_1}\cdots h_{i_l}g_2^{-1})$$
and 
$$\tau'(g_1\varphi(h)g_3^{-1})=\tau'(g_3x)+\tau'(k_{i_1}\cdots k_{i_l}g_4^{-1}).$$
By assumption, we are left with proving that 
$\tau(p_{i_1}^{-1}h_{i_1}\cdots h_{i_l}g_2^{-1})=\tau'(\overline{p_{i_1}}^{-1}k_{i_1}\cdots k_{i_l}g_4^{-1})$, respectively that 
$\tau(h_{i_1}\cdots h_{i_l}g_2^{-1})=\tau'(k_{i_1}\cdots k_{i_l}g_4^{-1})$. Further expanding the ends of the respective words, an analogous argument reduces the proof to showing that 
$$\tau(p_{i_1}^{-1}h_{i_1}\cdots h_{i_l}p_{i_l}^{-1})=\tau'(\overline{p_{i_1}}^{-1}k_{i_1}\cdots k_{i_l}\overline{p_{i_l}}^{-1}),$$
where we allow $p_{i_1}$ or $p_{i_l}$ to be the trivial word. This equality follows directly from Remark \ref{R:rem-rips-1} and Lemma \ref{L:tau-2}, just as in the proofs of Lemma \ref{L:Rips-LO-1} and Lemma  \ref{L:perfect-LO-2}.
\end{proof} 

\begin{lem} \label{L:Rips-LO-3} Let $g\in F$. If, for all $i>0$, there are pieces $p_i$ such that $g$ is $p_i$-reduced with $h_i$, then there is $f\in F$ such that for every $\varepsilon\in\{\pm 1\}$, $a\in A$ and $x\in X$, 
$$\hbox{$ \tau(gp_i^{-1}a^{\varepsilon})=\tau'(f{\overline{p_i}}^{-1}{a'}^{\varepsilon})$ and $\tau(gx^{\varepsilon})=\tau'(f{x}^{\varepsilon})$.}$$
\end{lem}
\begin{proof}
Let $g$ be $p_i$-reduced with $h_i$, for all $i>0$. We need to consider several cases.

\emph{Case 1:} $g$ ends with a letter in $\mathcal{A}$, with $c$ or with $e$. Then let $f=g'$. The assertion holds by choice of the words and the orders.

\emph{Case 2:} $g=g_0x_i$, $p_i=x_i^{-1}$ and $gh_i^{-1}=g_0w_i^{-1}(a_1b_1)$. Then $p_j$ is trivial for all $j\not=i$. Let us denote by $t$ the last letter of $g_0$ so that $g=g_1t^{\delta}x_i$. 
 If $t<c$, we let $f=g_0'x_i$. Then the assertion holds by definition. 
Otherwise $t\geq c$. If $t <x_i$ we let $f=g_1'{t'}^{\delta}c^{\delta}x_i$. Then the assertion holds by Lemma \ref{L:tau}. 
Indeed, as $A<c\leq t$, 
$$\tau(gp_i^{-1}a^{\varepsilon})=\tau(g_1t^{\delta}a^{\varepsilon})=\tau(g_1t^{\delta}c^{\delta}a^{\varepsilon})=\tau'(g_1'{t'}^{\delta}c^{\delta}{a'}^{\varepsilon})=\tau'(f{\overline{p_i}}^{-1}{a'}^{\varepsilon}),$$
and, as $c<x_i$,
$$\tau(gx^{\varepsilon})=\tau(g_1t^{\delta}x_ix^{\varepsilon})=\tau(g_1t^{\delta}c^{\delta}x_ix^{\varepsilon})=\tau'(g_1'{t'}^{\delta}c^{\delta}x_ix^{\varepsilon})=\tau'(f{x}^{\varepsilon}).$$
Similarly, $\tau(ga^{\varepsilon})=\tau'(fa^{\varepsilon})$, hence,  $\tau(gp_j^{-1}a^{\varepsilon})=\tau'(f{\overline{p_j}}^{-1}{a'}^{\varepsilon})$, for all $j\not=i$. 

If $t>x_i$, we let $f=g_1'{t'}^{\delta}e^{\delta}x_i$. As $A<e$ and $e>x_i$, as before, the assertion holds by Lemma \ref{L:tau}. 
 If $t^{\delta}=x_i$, we let $f=g_1'y_ix_i=g_1't'^{\delta}x_i$. Then the assertion holds by definition.
 
\emph{Case 3:} $g=g_0t^{\delta}$, where $t\in Y \cup \{d_1,d_2,d_3\}$. Then $gh_i$ is reduced and $p_i$ trivial for all $i$. We also note that $t>A$. 
If $t<X$, then $\leq=\leq_2$ and we let $f=g_1'{t'}^{\delta}c^{\delta}$. As $c<X$ and $c>A$, the assertion follows by Lemma \ref{L:tau}. Indeed, 
$$
\tau(ga^{\varepsilon})=\tau(g_0t^{\delta}a^{\varepsilon})=\tau(g_0t^{\delta}c^{\delta}a^{\varepsilon})=\tau'(g_0'{t'}^{\delta}c^{\delta}{a'}^{\varepsilon})=\tau'(fa^{\varepsilon}), 
$$
 and
 $$
\tau(gx^{\varepsilon})=\tau(g_0t^{\delta}x^{\varepsilon})=\tau(g_0t^{\delta}c^{\delta}x^{\varepsilon})=\tau'(g_0'{t'}^{\delta}c^{\delta}{x}^{\varepsilon})=\tau'(fx^{\varepsilon}). 
$$ 

 Otherwise, $t>X$ and $\leq=\leq_1$. Then we let $f=g_1'{t'}^{\delta}e^{\delta}$. As $e>X$ and $e>A$, the assertion follows by Lemma \ref{L:tau} as before. 
 
\emph{Case 4:} $g$ ends with a letter $x_i$ but $|p_i|>1$. Then $g$ splits as $g=g_0a^{\delta}x_i$, for some $a\in A$. In this case we let $f=g_0'{a'}^{\delta}x_i$. The assertion then holds by definition.

This completes the proof of the lemma.      
\end{proof}

Lemma \ref{L:Rips-LO-3} has the following immediate consequence. 

\begin{cor} \label{C:Rips-LO-2} If $g\in F$ is $p$-reduced with $h_i$, for all $i>0$, then there is $f\in F$ such that
$$ 1\leq_1^g h \hbox{ if and only if } 1\leq_2^{f} \varphi(h).$$
\end{cor}

Let $p_i$ denote the piece that is maximal with the property of being a suffix of $h_i$. 

\begin{lem}\label{L:Rips-LO-4} Let $g\in F$ such that $(g^{-1},h_i) < |h_i|-|p_i|$, for all $1\leqslant i$. 
Then there is $f\in F$ such that 
$ 1\leq^g h \hbox{ if and only if } 1\leq'^f \varphi(h), \hbox{ for all $h\in H$.}$
\end{lem}
\begin{proof}  Let $g\in G_1$ such that $(g^ {-1},h_i)< |h_i|-|p_i|$, for all $1\leqslant i$. Note that this implies that the reduced form of $gh_i$ ends with $p_i$.  In particular, $gh_i$ is $p$-reduced with all $h_j\not=h_i^{-1}$.

If $g$ is $p$-reduced with all $h_i$, we are done by Corollary \ref{C:Rips-LO-2}. Thus let $g$ be not $p$-reduced with $h_i$. Then  it has to be $p$-reduced for all other $h_j$, $j\not=i$. There are 5 cases to consider. 

\emph{Case 1:} $h_i=w(a_1,a_2)x^{\varepsilon}$. Then $g=g_0u_0^{-1}(a_1,a_2)$, where $w=u_0u_1$. By assumption $u_1$ is a non-trivial word. Indeed, otherwise $p_i=x_i$. This implies that $(g^{-1},h_i)<|h_i|-|p_i|=|h_i|-1$, a contradiction. 
In this case  $gh_i=g_0u_1(a_1,a_2)x_i^{\varepsilon}$ and $k_i=\varphi(h_i)=w(b_1,b_2)x^{\varepsilon}$. \\
We let $f=g_0'u_0^{-1}(b_1,b_2)=g'$.

\emph{Case 2:} $h_i=w(a_1,a_2)x_i^{\varepsilon}y^{\delta}x_i^{-\varepsilon}v(a_1,a_2)$ and $g=g_0x_i^{-\varepsilon}w^{-1}(a_1,a_2)$, so that $gh_i=g_0y^{\delta}x_i^{-\varepsilon}v(a_1,a_2)$. Then $k_i=w(b_1,b_2)y_i^{\varepsilon}d_2^{\delta}y_i^{-\varepsilon}v(b_1,b_2)$. Below we define $f_0$ such that 
\begin{align}
 \tau(g_0x_i^{-\varepsilon})=\tau'(f_0y_i^{-\varepsilon}) \hbox{ and } \tau(g_0y^{\delta})=\tau'(f_0d_2^{\delta}). \label{E:rips-case-3a}
\end{align} 
We now explain how to choose $f_0$. We first split $g_0=g_1t^{\gamma}$ and let $f_0=g_0's^{\gamma}=g_1'{t'}^{\gamma}s^{\gamma}$, where $s$ is a letter chosen depending on $t$. Then, by Lemma \ref{L:tau-2}, 
\begin{align*}
 \tau(g_0x^{-\varepsilon})=\tau(g_1t^{\gamma}) +\tau(t^{\gamma}x_i^{-\varepsilon}) \hbox{ and } \tau(g_0y^{\delta})=\tau(g_1{t}^{\gamma})+\tau(t^{\gamma}{y}^{\delta}),
\end{align*}
and 
\begin{align*}
 \tau'(f_0x^{-\varepsilon})=\tau'(g_1'{t'}^{\gamma}s^{\gamma}) +\tau'(s^{\gamma}y_i^{-\varepsilon}) \hbox{ and } \tau'(f_0{d_2}^{\delta})=\tau'(g_1'{t'}^{\gamma}{s}^{\gamma})+\tau'(s^{\gamma}{d_2}^{\delta}).
\end{align*}
The first terms in these sums equal by Lemma \ref{L:tau}(1).  
Below we choose $s$, hence, $f_0$, such that 
$$
\hbox{$\tau(t^{\gamma}x_i^{-\varepsilon})=\tau'(s^{\gamma}y_i^{-\varepsilon})$ and $\tau(t^{\gamma}{y}^{\delta})=\tau'(s^{\gamma}{d_2}^{\delta})$.}
$$
This yields \eqref{E:rips-case-3a}. We then let $f=f_0y_i^{-\varepsilon}w^{-1}(b_1,b_2)$ so that $fk_i=f_0d_2^{\delta}y_i^{-\varepsilon}v(b_1,b_2)$. 

 The choice of $s$ depends on $t$ as follows. 
If $t< c$ we let $s={t'}$, and the claim follows as $t$ and $t'$ are smaller than $x_i$ and $y$. Otherwise, $t\geq c$. If $t<x_i$ and $t<y$ we let $s=c$ and the claim follows as $c<X$ and $c<Y$. \\
If $t<x_i$ and $t=y$, then $\leq=\leq_2$, $\leq'=\leq_1$ and $\delta=\gamma$. We let $s=d_2$. As $d_2<_1y_i$  and $\delta=\gamma$, the claim follows by definition of $\tau$. \\
If $t<x_i$ and $t>y$, then $\leq=\leq_2$ and $\leq'=\leq_1$. Then we let $s=d_3$. As $d_3<_1y_i$ and $d_3>d_2$, the claim follows. \\
If $t=x_i$ and $t<y$, then $\gamma=-\varepsilon$. We let $s=d_1$. As $\gamma=-\varepsilon$ and $d_1<d_2$, the claim follows by definition of $\tau$. \\
If $t=x_i$ and $t>y$, then $\leq=\leq_2$ and $\leq'=\leq_1$. We let $s={t'}=y_i$. As $y_i>_1d_2$ this implies the claim. \\
If $t>x_i$ and $t<y$, then $\leq=\leq_1$, $\leq'=\leq_2$ and we let $s=d_1$. As $d_1>_2y_i$ and $d_1<d_2$ the claim follows. \\
If $t>x_i$ and $t=y$, then $\gamma=\delta$. Then we let $s=e$. As $e>y_i$ and $\gamma=\delta$, the claim follows by definition of $\tau$.\\
If $t>x_i$ and $t>y$, we let $s=e$. As $e>y_i$ and $e>d_2$, the claim follows. 

\emph{Case 3:} $h_i=w(a_1,a_2)x_i^{\varepsilon}y^{\delta}x_i^{-\varepsilon}v(a_1,a_2)$ and $g=g_0y^{-\delta}x_i^{-\varepsilon}w^{-1}(a_1,a_2)$, so that $gh_i=g_0x_i^{-\varepsilon}v(a_1,a_2)$. Then $k_i=\varphi(h_i)=w(b_1,b_2)y_i^{\varepsilon}d_2^{\delta}y_i^{-\varepsilon}v(b_1,b_2)$. We fix $f_0$ such that 
\begin{align}
 \tau(g_0x_i^{-\varepsilon})=\tau'(f_0y_i^{-\varepsilon}) \hbox{ and } \tau(g_0y^{-\delta})=\tau'(f_0d_2^{-\delta}). \label{E:rips-case-4a}
\end{align} 
We then let $f=f_0d_2^{-\delta}y_i^{-\varepsilon}w^{-1}(b_1,b_2)$ so that $fk_i=f_0y_i^{-\varepsilon}v(b_1,b_2)$. The choice of $f_0$ is exactly as in Case 2. 

\emph{Case 4:} $h_i=w(a_1,a_2)x_i^{\varepsilon}y^{\delta}x_i^{-\varepsilon}v(a_1,a_2)$ and $g=g_0u_0^{-1}(a_1,a_2)x_i^{\varepsilon}y^{-\delta}x_i^{-\varepsilon}w^{-1}(a_1,a_2)$, where $v=u_0u_1$ and $u_1$ is a non-trivial word that ends with the piece $p_i$. Then $gh_i=g_0u_1(a_1,a_2)$ and  $k_i=\varphi(h_i)=w(b_1,b_2)y_i^{\varepsilon}d_2^{\delta}y_i^{-\varepsilon}v(b_1,b_2)$. In this case we let $f=g_0'u_0^{-1}(b_1,b_2)y_i^{\varepsilon}d_2^{-\delta}y_i^{-\varepsilon}w^{-1}(b_1,b_2)$.

\emph{Case 5:} none of the above. In this case we let $f=g'$. 

We now proof the assertion of the Lemma in Case 2. 
We first prove that $\tau(gh_ig^{-1})=\tau'(fk_if^{-1})$. To do so, we expand, using Lemma \ref{L:tau-2}, 
$$\tau(gh_ig^{-1})=\tau(g_0y^{\delta})+\tau(y^{\delta}x_i^{-\varepsilon})+\tau(x_i^{-\varepsilon}v(a_1,a_2)w(a_1,a_2)x_i^{\varepsilon})+\tau(x_i^{\varepsilon}g_0^{-1})$$
and 
$$ \tau'(fk_if^{-1})=\tau'(f_0d_2^{\delta})+\tau'(d_2^{\delta}y_i^{-\varepsilon})+\tau'(y_i^{-\varepsilon}v(b_1,b_2)w(b_1,b_2)y_i^{\varepsilon})+\tau(y_i^{\varepsilon}f_0^{-1}).$$
Note that $vw$ has length at least 1. The first and the last terms of these sums equal by \eqref{E:rips-case-3a}, the second terms by Remark \ref{R:rem-rips-1} and third terms equal by definition. Thus the two sums indeed equal. We also observe that $\omega(gh_ig^{-1})=\omega(fk_if^{-1})$. 

Next let $h=h_{i_1}\cdots h_{i_l}$, where either $l>1$ or $h_{i_1}\not= h_i$ and $h_{i_1}\not = h_i^{-1}$. If $l>1$, either $g$ is $p$-reduced with $h_{i_1}$ or $h_{i_1}=h_i$ and $gh_i$ is $p$-reduced with $h_{i_2}$. Similarly,  either $g$ is $p$-reduced with $h_{i_l}^{-1}$ or $h_{i_l}=h_i$ and $gh_i$ is $p$-reduced with $h_{i_{l-1}}^{-1}$. Otherwise $g$ is $p$-reduced with $h_{i_1}$ and $h_{i_1}^{-1}$. The assertion of the lemma thus follows from Lemma \ref{L:Rips-LO-2} given that the following hold for all pieces $q_j$ such that $h_j$ and $g$,  or $h_j$ and $gh_i$, respectively, are $q_j$-reduced:
\begin{align*}
& \tau(gq_ja^{\zeta})=\tau'(f\overline{q_j}a'^{\zeta})  \hbox{ and } \tau(gx^{\zeta})=\tau'(fx^{\zeta}), \\
& \tau(gh_iq_ja^{\zeta})=\tau'(fk_i\overline{q_j}a'^{\zeta})  \hbox{ and } \tau(gh_ix^{\zeta})=\tau'(fk_ix^{\zeta}),
\end{align*}
for any choice of $\zeta \in \{\pm 1\}$. 
We prove that $\tau(gh_iq_ja^{\zeta})=\tau'(fk_i\overline{q_j}a'^{\zeta})$. The proof of the other equalities is analogous and will not be given. We expand, using Lemma \ref{L:tau-2}:
$$
\tau(gh_iq_ja^{\zeta})=\tau(g_0y^{\delta})+\tau(y^{\delta}x_i^{-\varepsilon})+\tau(x_i^{-\varepsilon}v(a_1,a_2)q_ja^{\zeta})
$$
and
$$ \tau'(fk_i\overline{q_j}a'^{\zeta})=\tau'(f_0d_2^{\delta})+\tau'(d_2^{\delta}y_i^{-\varepsilon})+\tau'(y_i^{-\varepsilon}v(a_1,a_2)\overline{q_j}{a'}^{\zeta}).$$
The first summands of these two sums are equal by \eqref{E:rips-case-3a}, the second by Remark \ref{R:rem-rips-1}, and the third are equal by definition.  
This terminates the proof of the lemma in Case 2. The proof in Case~3 is completely analogous, using \eqref{E:rips-case-4a} instead of \eqref{E:rips-case-3a}. The proof in the other cases is less involved and follows a similar line of argument.  
\end{proof}

To prove that $\varphi^{-1}$ is compatible for $(\mathcal L,\mathcal L)$ we need analogues of Lemmas  \ref{L:Rips-LO-2}, \ref{L:Rips-LO-3} and \ref{L:Rips-LO-4} for $\varphi^{-1}$. We start with the analogue of Lemmas \ref{L:Rips-LO-2} and \ref{L:Rips-LO-3}. Their proofs are completely analogous to the proof of Lemma \ref{L:Rips-LO-2} and \ref{L:Rips-LO-3} and will thus be omitted. In fact, in the proofs it is sufficient to replace any occurrence of $a_1$ by $b_1$, $a_2$ by $b_2$, $a$ by $b$, $h_i$ by $k_i$ and vice versa. 
 
 \begin{lem}\label{L:Rips-I-LO-2} Let $k=k_{i_1}\cdots k_{i_l}$. Let $g_1,g_3$ and $g_2,g_4\in F$ such that $g_1$ is $p_{i_1}$-reduced with $k_{i_1}$, $g_2$ is $\overline{p_{i_1}}$-reduced with $h_{i_1}$, $g_3$ is $p_{i_l}$-reduced with $k_{i_l}^{-1}$  and $g_4$ $\overline{p_{i_l}}$-reduced with $h_{i_l}^{-1}$. If, for every $\varepsilon\in\{\pm 1\}$, $b\in A$ and $x\in X$, 
\begin{enumerate}
\item $ \tau(g_1p_{i_1}^{-1}b^{\varepsilon})=\tau'(g_2{\overline{p_{i_1}}}^{-1}{b'}^{\varepsilon})$ and $\tau(g_1x^{\varepsilon})=\tau'(g_2{x}^{\varepsilon})$, 
\item $ \tau(g_3p_{i_l}^{-1}b^{\varepsilon})=\tau'(g_4{\overline{p_{i_l}}}^{-1}{b'}^{\varepsilon}) $ { and }  $\tau(g_3x^{\varepsilon})=\tau'(g_4{x}^{\varepsilon})$, 
\end{enumerate}
 then 
$$\tau(g_1kg_3^{-1})=\tau'(g_2\varphi^{-1}(k)g_4^{{-1}}).$$
\end{lem} 

\begin{lem}\label{L:Rips-I-LO-3}  If $g\in F$ is $p_i$-reduced with all $k_i$, then there is $f\in F$ such that for every $\varepsilon\in\{\pm 1\}$, $b\in A$ and $x\in X$, 
$$\hbox{$ \tau(gp_i^{-1}b^{\varepsilon})=\tau'(f{\overline{p_i}}^{-1}{b'}^{\varepsilon})$ and $\tau(gx^{\varepsilon})=\tau'(f{x}^{\varepsilon})$.}$$
\end{lem}

 Finally, we prove an analogue of Lemma \ref{L:Rips-LO-4} for $\varphi^{-1}$.
 
\begin{lem}\label{L:Rips-I-LO-4} Let $g\in F$ such that $(g^{-1},k_i) < |k_i|-|p_i|$, for all $1\leqslant i$. 
Then there is $f\in F$ such that 
$ 1\leq^g k \hbox{ if and only if } 1\leq'^f \varphi^ {-1}(k), \hbox{ for all $k\in K$.}$
\end{lem}

\begin{proof} The proof is analogous to the proof of Lemma \ref{L:Rips-LO-3}, and there are 5 cases to consider. As the choice of $f$ is different in Case 2 and 3 we discuss these cases below. 

Let $p_i$ be a piece that is maximal with the property of being a suffix of $v_i$. Let $g\in G$ such that $(g^ {-1},k_i)< |k_i|-|p_i|$, for all $1\leqslant i$. Note that this implies that the reduced form of $gk_i$ ends with $p_i$.  In particular, $gk_i$ is $p$-reduced with all $k_j\not=k_i^{-1}$.

If $g$ is $p$-reduced with all $k_i$, we are done by Corollary \ref{C:Rips-LO-2}. Thus let $g$ be not $p$-reduced with $k_i$. Then  it has to be $p$-reduced for all other $k_j$, $j\not=i$. 

\emph{Case 1:} $k_i=w(b_1,b_2)x^{\varepsilon}$. Then $g=g_0u_0^{-1}(b_1,b_2)$, where $w=u_0u_1$. By assumption $u_1$ is a non-trivial word. Indeed, otherwise $p_i=x_i$. On the other hand $(g^{-1},k_i)<|k_i|-|p_i|=|k_i|-1$, a contradiction. 
In this case  $gk_i=g_0u_1(b_1,b_2)x_i^{\varepsilon}$ and $h_i=\varphi^{-1}(k_i)=w(a_1,a_2)x^{\varepsilon}$. \\
We let $f=g_0'u_0^{-1}(a_1,a_2)$.

\emph{Case 2:} $k_i=w(b_1,b_2)y_i^{\varepsilon}d_2^{\delta}y_i^{-\varepsilon}v(b_1,b_2)$ and $g=g_0y_i^{-\varepsilon}w^{-1}(b_1,b_2)$, so that $gk_i=g_0d_2^{\delta}y_i^{-\varepsilon}v(b_1,b_2)$. Then $h_i=w(a_1,a_2)x_i^{\varepsilon}y^{\delta}x_i^{-\varepsilon}v(a_1,a_2)$, for some $y$. Below we define $f_0$ such that 
\begin{align}
 \tau(g_0y_i^{-\varepsilon})=\tau'(f_0x_i^{-\varepsilon}) \hbox{ and } \tau(g_0d_2^{\delta})=\tau'(f_0y^{\delta}). \label{E:rips-I-case-3a}
\end{align} 
We now explain how to choose $f_0$. We first split $g_0=g_1t^{\gamma}$ and let $f_0=g_0's^{\gamma}=g_1'{t'}^{\gamma}s^{\gamma}$, where $s$ is a letter chosen depending on $t$. Then, by Lemma \ref{L:tau-2}, 
\begin{align*}
 \tau(g_0y_i^{-\varepsilon})=\tau(g_1t^{\gamma}) +\tau(t^{\gamma}y_i^{-\varepsilon}) \hbox{ and } \tau(g_0d_2^{\delta})=\tau(g_1{t}^{\gamma})+\tau(t^{\gamma}{d_2}^{\delta}),
\end{align*}
and 
\begin{align*}
 \tau'(f_0x_i^{-\varepsilon})=\tau'(g_1'{t'}^{\gamma}s^{\gamma}) +\tau'(s^{\gamma}x_i^{-\varepsilon}) \hbox{ and } \tau'(f_0{y}^{\delta})=\tau'(g_1'{t'}^{\gamma}{s}^{\gamma})+\tau'(s^{\gamma}{y}^{\delta}).
\end{align*}
The first terms in these sums equal by Lemma \ref{L:tau}(1).  
Below we choose $s$, hence, $f_0$, such that 
$$
\hbox{$\tau(t^{\gamma}y_i^{-\varepsilon})=\tau'(s^{\gamma}x_i^{-\varepsilon})$ and $\tau(t^{\gamma}{d_2}^{\delta})=\tau'(s^{\gamma}{y}^{\delta})$.}
$$
 This yields \eqref{E:rips-I-case-3a}. We then let $f=f_0x_i^{-\varepsilon}w^{-1}(a_1,a_2)$ so that $fh_i=f_0y^{\delta}x_i^{-\varepsilon}v(a_1,a_2)$. 

 The choice of $s$ depends on $t$ as follows. 
If $t< c$ we let $s={t'}$, and the claim follows as $t$ and $t'$ are smaller than $X$ and $Y$. Otherwise, $t\geq c$. If $t<y_i$ and $t<d_2$ we let $s=c$ and the claim follows as $c<X$ and $c<Y$. \\
If $t<y_i$ and $t=d_2$, then $\leq=\leq_1$, $\leq'=\leq_2$ and $\delta=\gamma$. We let $s=c$. As, $c<_2x_i$ and $\delta=\gamma$, the claim follows by definition of $\tau$. \\
If $t<y_i$ and $t>d_2$, then $\leq=\leq_1$ and $\leq'=\leq_2$. Then we let $s=d_3$. As $d_3<_2x_i$ and $d_3>_2y$, the claim follows. \\
If $t=y_i$ and $t<d_2$, then $\leq=\leq_2$, $\leq'=\leq_1$ and $\gamma=-\varepsilon$. We let $s=d_2$. As $\gamma=-\varepsilon$ and $d_2<_1y$, the claim follows by definition of $\tau$. \\
If $t=y_i$ and $t>d_2$, then $\leq=\leq_1$, $\leq'=\leq_2$ and $\gamma=-\varepsilon$.  We let $s={t'}=x_i$. As $x_i>_2d_2$ this implies the claim. \\
If $t>y_i$ and $t<d_2$, then $\leq=\leq_2$, $\leq'=\leq_1$ and we let $s=d_2$. As $d_2>_1x_i$ and $d_2<_1y$ the claim follows. \\
If $t>y_i$ and $t=d_2$, then $\gamma=\delta$. We let $s=e$. As $e>x_i$ and $\gamma=\delta$, the claim follows by definition of $\tau$.\\
If $t>y_i$ and $t>d_2$, we let $s=e$. As $e>x_i$ and $e>y$, the claim follows. 

\emph{Case 3:} $k_i=w(b_1,b_2)y_i^{\varepsilon}d_2^{\delta}y_i^{-\varepsilon}v(b_1,b_2)$ and $g=g_0d_2^{-\delta}y_i^{-\varepsilon}w^{-1}(b_1,b_2)$, so that $gk_i=g_0y_i^{-\varepsilon}v(b_1,b_2)$. Then $h_i=\varphi^{-1}(k_i)=w(a_1,a_2)x_i^{\varepsilon}y^{\delta}x_i^{-\varepsilon}v(a_1,a_2)$, for some $y$. We fix $f_0$ such that 
\begin{align}
 \tau(g_0y_i^{-\varepsilon})=\tau'(f_0x_i^{-\varepsilon}) \hbox{ and } \tau(g_0d_2^{-\delta})=\tau'(f_0y^{-\delta}). \label{E:rips-I-case-4a}
\end{align} 
We then let $f=f_0y^{-\delta}x_i^{-\varepsilon}w^{-1}(a_1,a_2)$ so that $fh_i=f_0x_i^{-\varepsilon}v(a_1,a_2)$. The choice of $f_0$ is exactly as in Case 2. 

\emph{Case 4:} $k_i=w(b_1,b_2)y_i^{\varepsilon}d_2^{\delta}y_i^{-\varepsilon}v(b_1,b_2)$ and $g=g_0u_0^{-1}(b_1,b_2)y_i^{\varepsilon}d_2^{-\delta}y_i^{-\varepsilon}w^{-1}(b_1,b_2)$, where $v=u_0u_1$ and $u_1$ is a non-trivial word that ends with the piece $p_i$. Then $gk_i=g_0u_1(b_1,b_2)$ and  $h_i=\varphi^{-1}(k_i)=w(a_1,a_2)x_i^{\varepsilon}y^{\delta}x_i^{-\varepsilon}v(a_1,a_2)$. In this case we let $f=g_0'u_0^{-1}(a_1,a_2)x_i^{\varepsilon}y^{-\delta}x_i^{-\varepsilon}w^{-1}(a_1,a_2)$.

\emph{Case 5:} none of the above. In this case we let $f=g'$. 

The remainder of the proof is analogous to the proof of Lemma \ref{L:Rips-LO-3}, with only minor changes, and will thus not be given. 
\end{proof} 

Finally, we have the following. Its proof is the same as the proof of Lemma \ref{L:perfect-LO-4}, where Lemma \ref{L:perfect-LO-3} is replaced by Lemma \ref{L:Rips-LO-4} or Lemma \ref{L:Rips-I-LO-4} respectively, and will thus be skipped. 
\begin{lem}\label{L:Rips-LO-5} Let $g\in F$ such that $(g^{-1},h_i)\geqslant |h_i|-|p_i|$, for some $1\leq i$. 
Then there is $f\in F$ such that 
$ 1\leq_1^g h \hbox{ if and only if } 1\leq_2^f \varphi(h), \hbox{ for all $h\in H$.}$ 
The analogous statement holds for $\varphi^{-1}$. 
\end{lem} 

\begin{proof}[Proof of Theorem \ref{IT:Rips}] In view of Proposition \ref{P:Rips} it remains to prove that $G$ is left-orderable. To this end we note that $\varphi$ is compatible for $(\mathcal L,\mathcal L)$ by Lemmas \ref{L:Rips-LO-4} and \ref{L:Rips-LO-5}, and  $\varphi^{-1}$ is compatible for $(\mathcal L,\mathcal L)$ by Lemmas \ref{L:Rips-I-LO-4} and \ref{L:Rips-LO-5} respectively. Thus $\varphi$ is compatible and the assumptions of Theorem \ref{T:combination-HNN} are satisfied, hence, the claim holds. \end{proof} 
 
 \subsection{Embedding groups that are not locally indicable} Finally, we argue that the groups $G$ and $N$ in the Rips construction can be made not locally indicable (Proposition \ref{IP:not-locally-indicable}). To this end, let $P=\langle e_1,e_2,e_3,e_4\rangle$ be a left-orderable and perfect group given by Proposition \ref{P:perfect-lo}. We now alter the construction of the group $G$ in order to embed $P$ in $N$. The groups $G$ and $N$ so obtained are then not locally indicable, as $P$ is a finitely generated subgroup that is not indicable.
 
  As before let $F$ be the free group freely generated by 
  $$
  \langle a_1,a_2,b_1,b_2,x_1,\ldots x_n,d_1,d_2,d_3,y_1,\ldots,y_n,e\rangle,
  $$
  and let 
  $$\overline{F}=F*P=\langle a_1,a_2,b_1,b_2,c,x_1,\ldots x_n,d_1,d_2,d_3,y_1,\ldots,y_n,e,e_1,e_2,e_3,e_4\rangle.$$ 
 As before, we let $z_1=a_1$, $z_2=a_2$, $z_3=b_1$, $z_4=b_2$, $z_5=c$, $z_6=d_1$, $z_7=d_2$, $z_8=d_3$, $z_9=e$, and let $z_{10}=e_1$, $z_{11}=e_2$, $z_{12}=e_3$ and $z_{13}=e_4$. Then we define the elements $h_{2i}$ and $k_{2i}$ by \eqref{rips-relator-1}-\eqref{rips-relator-7}. That is, in comparison with the previous section, we add words $h_{j_{i,l}}$ and $k_{j_{i,l}}$ of type \eqref{rips-relator-5} and \eqref{rips-relator-6} for $z_l=z_l'=e_1$, $e_2$, $e_3$ or $e_4$ respectively.  Then we let $H$ be the subgroup freely generated by $h_2,h_4,\ldots, h_{2_i}, \ldots$ and $K$ the subgroup freely generated by $k_2,k_4,\ldots, k_{2_i}, \ldots$ respectively. As before we denote by $\varphi$ the natural isomorphism from $H$ to $K$ that sends $h_{2i}$ to $k_{2i}$, and let 
 $$
 G=\overline{F}*_H=\langle \overline{F},q \mid h=q\varphi(h)q^{-1}, \hbox{ for all $h\in H$} \rangle.
 $$
  Let $N$ be the subgroup of $G$ generated by $q$, $z_1$, $\ldots$, $z_{13}$, $y_1$, $\ldots$, $y_n$. 
 As in the proof of Proposition \ref{P:Rips}, $N$ is a normal subgroup of $G$ and equal to the kernel of the natural projection from $G$ onto $Q$. We also note that $P$ embeds into $G$, hence, it also embeds into $N$. Moreover, we may assume that $ h_2,h_4,\ldots, h_{2_i}, \ldots$, $k_2,k_4,\ldots, k_{2_i}, \ldots$ and the relators of $P$  are small cancellation words, hence, that $G$ is a small cancellation group itself. It remains to argue that $G$ is left-orderable. 
 
To this end we use explicit left-orders on the free product $\overline{F}$, as given by \cite{warren_orders_2020}, see Section~\ref{S:free-product}. Let $\leq_1$ and $\leq_2$ be the left-orders on $F$ already defined in the previous section. Let $\leq_P$ be a left-order on $P$. This yields two weight functions $\overline{\tau}_1$ and $\overline{\tau}_2$, and respective left-orders $\overline{\leq}_1$ and $\overline{\leq}_2$ on $\overline{F}$, see Section \ref{S:free-product} and Theorem \ref{T:left-order-free-product}. As before,$\overline{\leq}$ denotes one of these two orders, and $\overline{\leq}'$ the other one. We let $\overline{\mathcal L}$ be the normal closure of $\{\overline{\leq}_1,\overline{\leq}_2\}$ in $\overline{F}$. 
 
To prove that $\varphi$ is compatible for $(\overline{\mathcal{L}},\overline{\mathcal{L}})$, we first explain why Lemma \ref{L:Rips-LO-4} holds. If $g\in \overline F$ we denote by $g'$ the element obtained by replacing all syllables $s$ in $F$  by $s'=\phi(s)$. Let $g\in \overline{F}$ such that $(g^{-1},h_i)<|h_i|-|p_i|$, for all $i\geqslant 1$. Let $h=h_{i1}\cdots h_{i_l}$. First let $g\in \overline F$ be of syllable length $>1$. If $g$ ends with a syllable in $P$, then $ 1\; \overline\leq\; ghg^{-1}$ if, and only if, $1\;\overline\leq{}' \; g'\varphi(h)g'^{-1}$ by definition of the weight functions and orders on $\overline{F}$. 
 Otherwise, we split $g=g_1g_2$, where $g_1$ ends with a syllable in $P$, and $g_2\in F$. If there is $f_2\in F$ such that  $1\; \overline\leq^{g_2} h $ implies that  $1\;\overline{\leq}{}'^{f_2}  k$, then we let $f_1=g_1'$ and $f=f_1f_2$. Now $ 1 \; \overline\leq^{g_1}  g_2hg_2^{-1}$ if, and only if, $1\; \overline\leq{}'^{f_1}  f_2\varphi(h)f_2^{-1}$. Equivalently, $ 1\; \overline\leq \; ghg^{-1}$ if, and only if, $1\; \overline\leq{}'\; g'\varphi(h)g'^{-1}$. 

 This reduces the proof to the case that $g\in F$. 
 In this case, the assertion of Lemma \ref{L:Rips-LO-2} can be replaced by $\overline{\tau}(g_1hg_3^{-1})=\overline{\tau}'(g_2\varphi(h)g_4^{-1})$.  Indeed,  under the assumptions of lemma, the $r$-th syllable in $g_1hg_3^{-1}$ is positive if and only if the $r$-the syllable in $g_2\varphi(h)g_4^{-1}$ is positive, for every $r>0$. The proof of this claim is similar to the proof of Lemma \ref{L:Rips-LO-2}. We note that Lemma \ref{L:Rips-LO-3} remains valid; and its proof remains unchanged. This then yields Corollary \ref{C:Rips-LO-2} for $\overline{\leq}$ and $\overline{\leq}{}'$.  Similarly, Lemma \ref{L:Rips-LO-4} for $\overline{\leq}$ and $\overline{\leq}{}'$ now follows as before, by choosing $f$ and verifying the assumptions of Lemma \ref{L:Rips-LO-2} in the 5 cases of the proof. 

We now conclude Lemma \ref{L:Rips-LO-5} as before, and thus that $\varphi$ is compatible for $(\overline{\mathcal{L}},\overline{\mathcal{L}})$. Similarly, one proves that $\varphi^{-1}$ is compatible with $(\overline{\mathcal{L}},\overline{\mathcal{L}})$. Theorem \ref{T:combination-HNN} now implies that $G$ is left-orderable. This finishes the proof of Proposition \ref{IP:not-locally-indicable}.

% References 
\bibliographystyle{alpha}
\bibliography{orders-rips}

\begin{thebibliography}{HLNR21}

\bibitem[ADL11]{antolin_non-orientable_2011}
Yago Antol{\'{\i}}n, Warren Dicks, and Peter~A. Linnell.
\newblock Non-orientable surface-plus-one-relation groups.
\newblock {\em J. Algebra}, 326(1):4--33, 2011.

\bibitem[Ago13]{agol_virtual_2013}
Ian Agol.
\newblock The virtual {Haken} conjecture (with an appendix by {Ian} {Agol},
  {Daniel} {Groves} and {Jason} {Manning}).
\newblock {\em Doc. Math.}, 18:1045--1087, 2013.

\bibitem[AM15]{antolin_tits_2015}
Yago Antol{\'{\i}}n and Ashot Minasyan.
\newblock Tits alternatives for graph products.
\newblock {\em J. Reine Angew. Math.}, 704:55--83, 2015.

\bibitem[Are22]{arenas_cubical_2022}
Macarena Arenas.
\newblock A cubical rips construction.
\newblock {\em arXiv preprint arXiv:2202.01048}, 2022.

\bibitem[AS23]{arzhantseva_rips_2023}
Goulnara Arzhantseva and Markus Steenbock.
\newblock Rips construction without unique product.
\newblock {\em Pac. J. Math.}, 322(1):1--9, 2023.

\bibitem[Ber90]{bergman_ordering_1990}
George Bergman.
\newblock Ordering coproducts of groups and semigroups.
\newblock {\em J. Algebra}, 133(2):313--339, 1990.

\bibitem[BG09]{bludov_word_2009}
V.~V. Bludov and A.~M.~W. Glass.
\newblock Word problems, embeddings, and free products of right-ordered groups
  with amalgamated subgroup.
\newblock {\em Proc. Lond. Math. Soc. (3)}, 99(3):585--608, 2009.

\bibitem[BMS94]{baumslag_unsolvable_1994}
G.~Baumslag, C.~F.~III Miller, and H.~Short.
\newblock Unsolvable problems about small cancellation and word hyperbolic
  groups.
\newblock {\em Bull. Lond. Math. Soc.}, 26(1):97--101, 1994.

\bibitem[Bow98]{bowditch_continuously_1998}
B.~H. Bowditch.
\newblock Continuously many quasiisometry classes of {{\(2\)}}-generator
  groups.
\newblock {\em Comment. Math. Helv.}, 73(2):232--236, 1998.

\bibitem[Bro80]{brodskii_equations_1980}
S.~D. Brodski\u{\i}.
\newblock Equations over groups and groups with one defining relation.
\newblock {\em Uspekhi Mat. Nauk}, 35(4(214)):183, 1980.

\bibitem[BRW05]{boyer_orderable_2005}
Steven Boyer, Dale Rolfsen, and Bert Wiest.
\newblock Orderable 3-manifold groups.
\newblock {\em Ann. Inst. Fourier}, 55(1):243--288, 2005.

\bibitem[BW01]{bridson_malnormality_2001}
Martin~R. Bridson and Daniel~T. Wise.
\newblock Malnormality is undecidable in hyperbolic groups.
\newblock {\em Isr. J. Math.}, 124:313--316, 2001.

\bibitem[BW05]{bumagin_every_2005}
Inna Bumagin and Daniel~T. Wise.
\newblock Every group is an outer automorphism group of a finitely generated
  group.
\newblock {\em J. Pure Appl. Algebra}, 200(1-2):137--147, 2005.

\bibitem[Chi11]{chiswell_right_2011}
I.~M. Chiswell.
\newblock Right orderability and graphs of groups.
\newblock {\em J. Group Theory}, 14(4):589--601, 2011.

\bibitem[Cou22]{coulon_examples_2022}
R{\'e}mi Coulon.
\newblock Examples of groups whose automorphisms have exotic growth.
\newblock {\em Algebr. Geom. Topol.}, 22(4):1497--1510, 2022.

\bibitem[CR16]{clay_ordered_2016}
Adam Clay and Dale Rolfsen.
\newblock {\em Ordered groups and topology}, volume 176 of {\em Grad. Stud.
  Math.}
\newblock Providence, RI: American Mathematical Society (AMS), 2016.

\bibitem[DGP11]{dahmani_random_2011}
Fran{\c{c}}ois Dahmani, Vincent Guirardel, and Piotr Przytycki.
\newblock Random groups do not split.
\newblock {\em Math. Ann.}, 349(3):657--673, 2011.

\bibitem[DNR14]{deroin_groups_2014}
Bertrand Deroin, Andr{\'e}s Navas, and Crist{\'o}bal Rivas.
\newblock Groups, orders, and dynamics.
\newblock {\em arXiv preprint arXiv:1408.5805}, 2014.

\bibitem[D{\v{S}}20]{warren_orders_2020}
Warren Dicks and Zoran {\v{S}}uni{\'c}.
\newblock Orders on trees and free products of left-ordered groups.
\newblock {\em Can. Math. Bull.}, 63(2):335--347, 2020.

\bibitem[GR20]{gitk_growth_2020}
Rita Gitik and Eliyahu Rips.
\newblock On growth of double cosets in hyperbolic groups.
\newblock {\em Int. J. Algebra Comput.}, 30(6):1161--1166, 2020.

\bibitem[Hem90]{hempel_one-relator_1990}
John Hempel.
\newblock One-relator surface groups.
\newblock {\em Math. Proc. Camb. Philos. Soc.}, 108(3):467--474, 1990.

\bibitem[HL19]{hyde_finitely_2019}
James Hyde and Yash Lodha.
\newblock Finitely generated infinite simple groups of homeomorphisms of the
  real line.
\newblock {\em Invent. Math.}, 218(1):83--112, 2019.

\bibitem[HLNR21]{hyde_uniformly_2021}
James Hyde, Yash Lodha, Andr{\'e}s Navas, and Crist{\'o}bal Rivas.
\newblock Uniformly perfect finitely generated simple left orderable groups.
\newblock {\em Ergodic Theory Dyn. Syst.}, 41(2):534--552, 2021.

\bibitem[HW08]{haglund_special_2008}
Fr{\'e}d{\'e}ric Haglund and Daniel~T. Wise.
\newblock Special cube complexes.
\newblock {\em Geom. Funct. Anal.}, 17(5):1551--1620, 2008.

\bibitem[JW22]{Jankiewicz_cubulating_2022}
Kasia Jankiewicz and Daniel Wise.
\newblock Cubulating small cancellation free products.
\newblock {\em Indiana Univ. Math. J.}, 71(4):1397--1409, 2022.

\bibitem[KM23]{khukhro_unsolved_2023}
E.~I. Khukhro and V.~D. Mazurov.
\newblock Unsolved problems in group theory. the kourovka notebook, 2023.

\bibitem[LS01]{lyndon_combinatorial_2001}
Roger~C. Lyndon and Paul~E. Schupp.
\newblock {\em Combinatorial group theory.}
\newblock Class. Math. Berlin: Springer, reprint of the 1977 ed. edition, 2001.

\bibitem[MBT20]{mattebon_groups_2020}
Nicol{\'a}s Matte~Bon and Michele Triestino.
\newblock Groups of piecewise linear homeomorphisms of flows.
\newblock {\em Compos. Math.}, 156(8):1595--1622, 2020.

\bibitem[MOW21]{minasyan_quasi_2021}
A.~Minasyan, D.~Osin, and S.~Witzel.
\newblock Quasi-isometric diversity of marked groups.
\newblock {\em J. Topol.}, 14(2):488--503, 2021.

\bibitem[Orl17]{orlef_random_2017}
Damian Orlef.
\newblock Random groups are not left-orderable.
\newblock {\em Colloq. Math.}, 150(2):175--185, 2017.

\bibitem[Orl21]{orlef-non-orderability_2021}
Damian Orlef.
\newblock Non-orderability of random triangular groups by using random 3cnf
  formulas.
\newblock {\em Bull. Lond. Math. Soc.}, 53(5):1324--1332, 2021.

\bibitem[Pri83]{Pride_some_1983}
Stephen~J. Pride.
\newblock Some finitely presented groups of cohomological dimension two with
  property ({FA}).
\newblock {\em J. Pure Appl. Algebra}, 29:167--168, 1983.

\bibitem[Rip82]{rips_subgroups_1982}
E.~Rips.
\newblock Subgroups of small cancellation groups.
\newblock {\em Bull. Lond. Math. Soc.}, 14:45--47, 1982.

\bibitem[Riv12]{rivas_left-orderings_2012}
Crist{\'o}bal Rivas.
\newblock Left-orderings on free products of groups.
\newblock {\em J. Algebra}, 350(1):318--329, 2012.

\bibitem[RW01]{rolfsen-free-2001}
Dale Rolfsen and Bert Wiest.
\newblock Free group automorphisms, invariant orderings and topological
  applications.
\newblock {\em Algebr. Geom. Topol.}, 1:311--319, 2001.

\bibitem[{\v{S}}un13]{sunic_explicit_2013}
Zoran {\v{S}}uni{\'c}.
\newblock Explicit left orders on free groups extending the lexicographic order
  on free monoids.
\newblock {\em C. R., Math., Acad. Sci. Paris}, 351(13-14):507--511, 2013.

\bibitem[Wis03]{wise_residually_2003}
Daniel~T. Wise.
\newblock A residually finite version of {Rips}'s construction.
\newblock {\em Bull. Lond. Math. Soc.}, 35(1):23--29, 2003.

\bibitem[Wis04]{wise_cubulating_2004}
D.~T. Wise.
\newblock Cubulating small cancellation groups.
\newblock {\em Geom. Funct. Anal.}, 14(1):150--214, 2004.

\end{thebibliography}

\end{document}